\documentclass[11pt]{article}

\usepackage{amsfonts} \usepackage{amssymb} \usepackage{amsmath}
\usepackage{amsthm} \usepackage{color}  \usepackage{amsmath,amscd}
\usepackage{graphicx}
\usepackage{psfrag}
\usepackage{pstricks}

\newenvironment{enumalph}
{\begin{enumerate}}
{\end{enumerate}}

\newtheorem{theorem}{Theorem}

\newtheorem{definition}[theorem]{Definition}
 
\newtheorem{lemma}[theorem]{Lemma} 

\newtheorem{proposition}[theorem]{Proposition}

\def\H{\mathbb{H}}
\def\T{\mathcal{T}}
\def\X{\mathcal{X}}

\DeclareMathOperator{\pslc}{PSL(2, \mathbb{C})}
\DeclareMathOperator{\pslr}{PSL(2, \mathbb{R})}
\DeclareMathOperator{\Hom}{Hom}

\DeclareMathOperator{\out}{Out}

\DeclareMathOperator{\inj}{inj}
\DeclareMathOperator{\thin}{thin}
\DeclareMathOperator{\thick}{thick}
\DeclareMathOperator{\Mod}{Mod}
\DeclareMathOperator{\inte}{int}

\DeclareMathOperator{\homeo}{Homeo}
\DeclareMathOperator{\id}{id}
\title{Dynamics on the $\pslc$-character variety of a compression body \thanks{Partially supported by NSF RTG grant DMS 0602191}}
\author{Michelle Lee}
\date{}

\begin{document}
\maketitle

\begin{abstract}
   \centering
   \begin{minipage}{0.65\textwidth}
Let $M$ be a nontrivial compression body without toroidal boundary components.  Let $\mathcal{X}(M)$ be the $\pslc$-character variety of $\pi_1(M)$.  We examine the dynamics of the action of $\out(\pi_1(M))$ on $\mathcal{X}(M),$ and in particular, we find an open set on which the action is properly discontinuous that is strictly larger than the interior of the deformation space of marked hyperbolic $3$-manifolds homotopy equivalent to $M$.  \end{minipage}
\end{abstract}

In this paper we use the deformation theory of hyperbolic $3$-manifolds to study the dynamics of $\out(\pi_1(M))$ on the $\pslc$-character variety of $\pi_1(M)$ when $M$ is a nontrivial compression body without toroidal boundary components.  In particular, we find a domain of discontinuity for the action that is strictly larger than the previously known domain of discontinuity.

The study of $\out(\pi_1(M))$ acting on character varieties or representation varieties is a blooming field of study.  One motivation comes from the classical result that the mapping class group of a closed oriented surface $S$ of genus at least two acts properly discontinuously on $\T(S)$ the Teichm\"uller space of $S$.  Teichm\"uller space $\T(S)$ is a component of the representation variety $\Hom(\pi_1(S), \pslr)/\pslr$ and together with $\T(\bar S)$ the Teichm\"uller space of $S$ with the opposite orientation, form the set of discrete and faithful representations.  The group $\out(\pi_1(S))$ acts properly discontinuously on $\T(S) \sqcup \T(\bar S)$ and Goldman conjectured that the action on the remaining components is ergodic.  The so-called higher Teichm\"uller spaces, which are analogies of Teichm\"uller space for higher rank Lie groups, also form domains of discontinuity (see, for example, \cite{lab2}, \cite{wie}, \cite{har-str}).

A compression body is the boundary connect sum of a $3$-ball, a collection of $I$-bundles over closed surfaces and a handlebody where the other components are connected to the $3$-ball along disjoint discs.  The $\pslc$-character variety of $\pi_1(M)$ is
$$\mathcal{X}(M) = \Hom(\pi_1(M), \pslc) /\!\!/ \pslc,$$
the quotient of $\Hom(\pi_1(M), \pslc)$ from geometric invariant theory.

The group $\out(\pi_1(M))$ acts on $\mathcal{X}(M)$ in the following way: an outer automorphism $[f]$ maps a representation $[\rho]$ to $[\rho$ $\circ f^{-1}]$.
Sitting inside $\X(M)$ is $AH(M)$ the space of conjugacy classes of discrete and faithful representations of $\pi_1(M)$ into $\pslc$.  It can also be thought of as the space of marked hyperbolic $3$-manifolds homotopy equivalent to $M.$  Using the parametrization of the interior of $AH(M)$ (see \cite{can-mcc} Chapter 7 for more details on this parametrization), it is well known that this action is properly discontinuous on the interior of $AH(M).$
In this paper we find a domain of discontinuity containing the interior of $AH(M)$ as well as some but not all points on $\partial AH(M)$ when $M$ is a nontrivial hyperbolizable compression body without toroidal boundary components.  Namely, we prove the following.

\begin{theorem} \label{compbody}
If $M$ is a nontrivial hyperbolizable compression body without toroidal boundary components, then there exists an open, $\out(\pi_1(M))$-invariant subset $\mathcal{SS}(M)$, called the set of separable-stable representations, in $\mathcal{X}(M)$ containing the interior of $AH(M)$ as well as points on $\partial AH(M)$ such that the action of $\out(\pi_1(M))$ is properly discontinuous on $\mathcal{SS}(M)$.
\end{theorem}

In proving the theorem we show that pinching a Masur domain curve or lamination on the boundary component are points in this domain of discontinuity.

Canary-Storm (\cite{can-sto}) showed that whenever $M$ has an primitive essential annulus the action of $\out(\pi_1(M))$ cannot be properly discontinuous on all of $AH(M).$  As compression bodies contain primitive essential annuli, one cannot hope to obtain a domain of discontinuity for this action containing all of $AH(M).$

Theorem \ref{compbody} is a generalization of a result by Minsky (\cite{min}) in the case when $M$ is a handlebody.  He introduced the notion of primitive-stable representations and showed that if $H_g$ is a genus $g$ hyperbolizable handlebody, then the set of primitive-stable representations, denoted $\mathcal{PS}(H_g)$, is a domain of discontinuity for the action of $\out(F_g)$ containing the interior of $AH(H_g)$ as well as points on $\partial AH(H_g)$, where $F_g$ is the free group on $g$ generators.  The set $\mathcal{SS}(H_g)$ of separable-stable representations, in the case when $M$ is a handlebody of genus $g,$ is contained in but not a priori equal to $\mathcal{PS}(H_g)$, although the two sets coincide on $AH(H_g)$.

The incompressible boundary case was resolved by Canary-Storm (\cite{can-sto}), Canary-Magid (\cite{can-mag}) and Lee (\cite{lee1}); they showed that in this case, there exists an open $\out(\pi_1(M))$-invariant set, containing the interior of $AH(M)$ and points on the boundary of $AH(M)$, on which $\out(\pi_1(M))$ acts properly discontinuously if and only if $M$ is not a trivial $I$-bundle over a closed orientable hyperbolic surface.

We end the introduction with a brief outline of the paper.  In Section \ref{background} we recall background material from topology and hyperbolic geometry that we will need.  In Section \ref{ssreps} we define the set of separable-stable representations and show that it is an open, $\out(\pi_1(M))$-invariant subset of $\mathcal{X}(M)$ containing the interior of $AH(M)$ on which $\out(\pi_1(M))$ acts properly discontinuously.  In Section \ref{boundarypoints} we find two types of points on $\partial AH(M)$ that are separable-stable; namely points $[\rho]$ whose associated hyperbolic manifold $N_\rho$ is homeomorphic to the interior of $M$ that satisfy one of the following two conditions: $N_\rho$ is geometrically finite with one cusp associated to a Masur domain curve or $N_\rho$ is purely hyperbolic with one geometrically infinite end corresponding to the compressible boundary component.  As the interior of $AH(M)$ consists of convex cocompact representations, these points lie on $\partial AH(M)$.  This will complete the proof of Theorem \ref{compbody}.  In Section \ref{boundarypoints} we make use of Otal's generalization of Whitehead graphs to compression bodies described in (\cite{ota}).  As it is difficult to procure, we give proofs of many of the results we use.  In Section \ref{nonhomeo}, we study further the structure of $\mathcal{SS}(M)$ and show that when $M$ is a large compression body, there exist connected components of $\mathcal{SS}(M)$ that coincide with components of the interior of $AH(M).$

\vspace{0.2in}
\textbf{Acknowledgements: } The results in this paper form part of the author's Ph.D. thesis under the guidance of Dick Canary.  The author warmly thanks Dick Canary for his valuable support and advice.

\section{Preliminaries} \label{background}

\subsection{Compression bodies} \label{compbodyback}
A \emph{compression body} is a compact, orientable, irreducible $3$-manifold $M$ with a boundary component $\partial_{ext}M$, called the \emph{exterior boundary}, whose inclusion induces a surjection $\pi_1(\partial_{ext}M) \rightarrow \pi_1(M)$.  The other boundary components are called \emph{interior boundary} components.  Equivalently, a compression body is a boundary connect sum of a $3$-ball, a collection of solid tori and a collection of trivial $I$-bundles over closed surfaces such that the other summands are attached to the $3$-ball along disjoint discs.

A compression body is \emph{trivial} if it is a trivial $I$-bundle over a closed surface.
A compression body is \emph{small} if there exists an essential, properly embedded disc $D$ such that $M-D$ is either two trivial $I$-bundles over closed surfaces or one trivial $I$-bundle over a closed surface; otherwise $M$ is a \emph{large} compression body.

The fundamental group of a compression body can be expressed as $G_1 * G_2 * \cdots * G_n$ where $G_i$ is isomorphic to a closed surface group for $1 \leq i \leq k$ and $G_j$ is infinite cyclic $k<j \leq n$.  By Grushko's theorem (\cite{gru}) and Kurosh's subgroup theorem (\cite{kur}) any other decomposition of the fundamental group into a free product, $H_1 * H_2 * \cdots H_m,$ where each factor is freely indecomposable, satisfies $n=m$ and $H_i \cong G_i$, up to re-ordering.

One can also think about splittings of $\pi_1(M)$ into free products geometrically.  Suppose that $D$ is a properly embedded essential disc in $M$ and $M_1$ and $M_2$ are the components of $M\backslash \mathcal{N}(D)$ where $\mathcal{N}(D)$ is a regular neighborhood of $D.$  If $M_i'=M_i \cup \mathcal{N}(D)$ then $\pi_1(M) \cong i_*(\pi_1(M_1')) * i_*(\pi_1(M_2'))$ where $i$ is inclusion and the basepoint is chosen to lie in $D.$  As the basepoint my change, this splitting is only well-defined up to conjugation.

As the following lemma shows, the converse is also true.
\begin{lemma}\label{mer}
Let $M$ be a compression body and $\pi_1(M)=H*K$ a nontrivial splitting of $\pi_1(M)$ into a free product.  Then, there exists a properly embedded disc $D$ realizing the splitting in the sense described above.

\end{lemma}

\begin{proof}
Recall that $M$ is a boundary connect sum of a collection of solid tori, a collection of trivial $I$-bundles and a $3$-ball such that the other summands are attached to the $3$-ball along disjoint discs.  The splitting corresponding to this connect sum, $G=G_1 * G_2 * \cdots * G_n,$ where $G_i$ is isomorphic to the fundamental group of a closed surface for $1 \leq i \leq k$ and $G_j$ is infinite cyclic for $k < j \leq n$ has the following property.  If $\sigma$ is any permutation of $\{1, \ldots, n\}$ then the splitting $G'*G''$ where $G'=G_{\sigma(1)} * \cdots * G_{\sigma(l)}$ and $G''=G_{\sigma(l+1)} * \cdots * G_{\sigma(n)}$ is realizable by an essential disc $\Delta_\sigma$.

A result of McCullough-Miller (\cite{mcc-mil} Corollary 5.3.3) describes the cosets of $\homeo^+(M)$, the group of orientation preserving homeomorphisms, in $\out(\pi_1(M))$ in the following way.  For each surface group factor $G_i$, let $f_i: G \rightarrow G$ be an automorphism such that $f_i|_{G_j}=\id|_{G_j}$ for $i \neq j$ and $f_i|_{G_i}$ is an orientation-reversing automorphism; notice that $f_i$ and $f_j$ commute for all $i, j \leq k$.  The cosets of $\homeo^+(M)$ in $\out(\pi_1(M))$ are $\{f_{i_1} \circ f_{i_2} \circ \cdots \circ f_{i_l} \cdot \homeo^+(M)\}$ for $1 \leq i_1 < i_2 < \cdots < i_l \leq k$.

Let $\sigma$ be a permutation of $\{1, \ldots, n\}$ such that $G_H=G_{\sigma(1)} * \cdots * G_{\sigma(l)}$ is isomorphic to $H$ and $G_K=G_{\sigma(l+1)} * \cdots * G_{\sigma(n)}$ is isomorphic to $K$.  Such a permutation must exist by the uniqueness of a decomposition of $G$ into a free product with freely indecomposable factors.  By the discussion above, there exists a disc $D'$ realizing the splitting $G=G_H * G_K$.  We can find an automorphism $\phi: G \rightarrow G$ such that $\phi(G_H)=H$ and $\phi(G_K)=K$.  Notice that the automorphisms $f_i$ do not affect the splitting of $G$.  By pre-composing with $f_i,$ if necessary, we can assume that $\phi|_{G_i}$ for $i \leq k$ is orientation-preserving.  Hence, $\phi$ is realizable by a homeomorphism $f_{\phi}$ and we can take $D$ to be the image of $f_\phi(D')$.
\end{proof}

\subsection{Hyperbolic geometry}
For this section, let $N$ denote a hyperbolic $3$-manifold, that is $N \cong \H^3 / \Gamma,$ where $\Gamma$ is a discrete group of orientation-preserving isometries of $\H^3.$
Given $\epsilon>0,$ the \emph{$\epsilon$-thin part} of $N$ is
$$N_{\thin(\epsilon)}=\{x \in N | \inj_N(x) < \epsilon\},$$
where $\inj_N(x)$ is the injectivity radius of $N$ at the point $x$.  The \emph{$\epsilon$-thick part} $N_{\thick(\epsilon)}$ is the complement of $N_{\thin(\epsilon)}.$  There exists a constant $\mu_3>0,$ called the \emph{Margulis constant,} such that for any hyperbolic $3$-manifold $N$ and any $\epsilon < \mu_3,$ each component of
$N_{\thin(\epsilon)}$ is one of the following (see  \cite{ben-pet}, Chp. D):
\begin{enumalph}
\item a metric neighborhood of a closed geodesic,
\item a parabolic cusp homeomorphic $S^1 \times \mathbb{R} \times (0, \infty),$ or
\item a parabolic cusp homeomorphic to $T \times (0, \infty),$ where $T$ is a torus.
\end{enumalph}
Let $N^0_\epsilon$ denote $N$ with all components of type (b) and (c) removed.

The \emph{convex core} $C(N)$ of $N$ is the smallest convex submanifold of $N$ such that the inclusion of $C(N)$ into $N$ is a homotopy equivalence.  The \emph{limit set} $\Lambda(\Gamma) \subset \partial \H^3$ is the smallest, closed $\Gamma$-invariant subset of $\partial \H^3$.  When $\Gamma$ is non-elementary, $C(N)$ is $CH(\Lambda(\Gamma))/\Gamma,$ where $CH(\Lambda(\Gamma))$ is the convex hull of the limit set $\Lambda(\Gamma)$.  $N$ is called \emph{convex cocompact} if $C(N)$ is compact.  $N$ is called \emph{geometrically finite} if $C(N) \cap N^0_\epsilon$ is compact and \emph{geometrically infinite} otherwise.

In general, when $\pi_1(N)$ is finitely generated there exists a compact submanifold $C,$ called the \emph{compact core}, whose inclusion induces a homotopy equivalence with $N$ (see \cite{sco}).   Moreover, $C$ can be chosen such that it intersects each component of the noncompact portion of $N_{\thin(\epsilon)}$ in a single incompressible annulus if the component is homeomorphic to $S^1 \times \mathbb{R} \times (0, \infty)$ or a single incompressible torus if the component is homeomorphic to $T \times (0, \infty)$ (see \cite{mcc}).  A compact core that intersects each component of the noncompact portions of $N_{\thin(\epsilon)}$ in this way is called a \emph{relative compact core}.

The \emph{domain of discontinuity} $\Omega(\Gamma)$ is the complement of $\Lambda(\Gamma)$ in $\partial \H^3$; it is the largest open set of $\partial \H^3$ on which $\Gamma$ acts properly discontinuously.  It can be uniquely endowed with a $\Gamma$-invariant hyperbolic metric, conformally equivalent to the metric induced by considering $\Omega(\Gamma)$ as a subset of $\partial \H^3 \cong \mathbb{C}P^1$.  The \emph{conformal boundary} of $N$ is $\partial_C N=\Omega(\Gamma)/ \Gamma$ a collection of hyperbolic surfaces obtained by taking the quotient of $\Omega(\Gamma)$ by $\Gamma.$  The conformal bordification of $N,$ $(\H^3 \cup \Omega(\Gamma))/\Gamma$ is homeomorphic to $C(N),$ except when $\Gamma$ is Fuchsian.

\subsubsection{Measured laminations}
Let $T$ be a closed hyperbolic surface.  A \emph{(geodesic) lamination} on $T$ is a closed subset $\lambda \subset T$ which is a union of disjoint simple geodesics.  A \emph{leaf} of $\lambda$ is a simple geodesic in $\lambda$.  A lamination $\lambda$ is \emph{minimal} if each half-leaf is dense in $\lambda.$  A \emph{measured lamination} is a pair $(\lambda, \nu)$ where $\lambda$ is a geodesic lamination and $\nu$ is a Borel measure on arcs transverse to $\lambda$ such that the support of $\nu$ is $\lambda$ and $\nu$ is invariant under isotopies of $T$ preserving $\lambda$.  Let $ML(T)$ denote the space of measured laminations on $T$ with the weak-$*$ topology on measures and let $PML(T)$ denote $(ML(T)-\{\emptyset\})/\mathbb{R}^+,$ the space of projective measured laminations.  Weighted simple closed geodesics are dense in $ML(T)$ (\cite{thu}, Proposition 8.10.7).  For two simple closed geodesics $\gamma_1$ and $\gamma_2$ the intersection number $i(\gamma_1, \gamma_2)$ is the number of points in $\gamma_1 \cap \gamma_2$.  This intersection number naturally extends to any two weighted simple closed geodesics and furthermore to a continuous map $i: ML(T) \times ML(T) \rightarrow \mathbb{R}_{\geq 0}$ (see \cite{bon1} Proposition 4.4).  A measured lamination $(\lambda, \mu)$ is \emph{filling} if for any other measured lamination $(\alpha, \nu)$ with different support, $i( (\lambda, \mu), (\alpha, \nu))$ is nonzero.

\subsubsection{Masur domain laminations}
For this section suppose that $M$ is a compression body.  A \emph{meridian} is a simple closed curve on $\partial M$ that is nontrivial in $\pi_1(\partial M)$ but is trivial in $\pi_1(M).$
Let $\mathcal{M}$ denote the set of meridians on $\partial M$ and let $\mathcal{M}'$ denote the closure of $\mathcal{M}$ in $PML(\partial_{ext}(M))$ (here we fix a convex cocompact hyperbolic structure on $M$ and consequently a hyperbolic structure on $\partial M$).
Let $$\mathcal{M}''=\{\lambda \in PML(\partial_{ext}(M)) | \text{  there exists a  } \nu \in M' \text{  such that  } i(\lambda, \nu) = 0\}.$$

\begin{definition}
If $M$ is not the connect sum of two trivial $I$-bundles over closed surfaces, then $\lambda$ in $PML(\partial_{ext}(M))$ lies in the \emph{Masur domain}, denoted $\mathcal{O}(M),$ if it has nonzero intersection number with every element of $\mathcal{M}'$.
If $M$ is a connect sum of two trivial $I$-bundles over closed surfaces, then an element $\lambda$ in $PML(\partial_{ext}(M))$ lies in $\mathcal{O}(M)$ if it has nonzero intersection number with every lamination in $\mathcal{M}''$.
\end{definition}

The Masur domain $\mathcal{O}(M)$ is an open set in $PML(\partial_{ext}(M))$ on which $\Mod(M)$, the group of isotopy classes of homeomorphisms of $\partial_{ext}(M)$ that extend to homeomorphisms of the compression body, acts properly discontinuously (\cite{mas} and \cite{ota} Proposition 1.4).  In particular, if $\lambda$ is the support of an element in $\mathcal{O}(M),$ then $\lambda$ intersects every essential annulus in $M$.  Otherwise, Dehn twists about an essential annulus missing $\lambda$ produce infinitely many elements in the stabilizer of $\lambda$ in $\Mod(M)$, which would contradict the proper discontinuity of the action.

For a general hyperbolizable $3$-manifold $M$ with compressible boundary, Lecuire described an extension of the Masur domain, called the set of \emph{doubly incompressible} laminations.  A measured lamination $(\lambda, \mu)$ is doubly incompressible if there exists a constant $\eta > 0$ such that $i(\lambda, \partial E) > \eta$ for any essential disc or essential annulus $E$.  When $M$ is a compression body, the pre-image of $\mathcal{O}(M)$ in $ML(S)$ is contained in $\mathcal{D}(M)$ (\cite{lec}, Lemma 3.1).

\subsection{Pleated surfaces} \label{psurf}
Pleated surfaces are one type of 1-Lipschitz maps from negatively curved surfaces into hyperbolic $3$-manifolds.

\begin{definition}
A \emph{pleated surface} in a hyperbolic $3$-manifold $N$ is a surface $S$ with a hyperbolic metric $\tau$ of finite area and a map $h: (S, \tau) \rightarrow N$ which takes rectifiable arcs in $S$ to rectifiable arcs of the same length in $N$ such that every point $x$ in $S$ lies in the interior of some geodesic arc that is mapped by $p$ to a geodesic arc in $N$.  The \emph{pleating locus} is the set of points in $S$ that lie in the interior of exactly one geodesic arc that is mapped to a geodesic arc in $N$.
\end{definition}
The pleating locus is a geodesic lamination that maps to a union of geodesics in $N$.  Although there can be many pleated surfaces in a fixed homotopy class realizing a geodesic lamination $\lambda,$ the image of $\lambda$ in $N$ is unique in that homotopy class (\cite{ceg}, Lemma I.5.3.5).  If $T$ is an incompressible boundary component of $N$ and $\lambda$ is a filling lamination on $T,$ then $\lambda$ can be realized as the pleating locus of a pleated surface homotopic to the inclusion of $T$ in $N$.  When $T$ is a compressible boundary component of $N$ and $\lambda \subset T$ is the support of a doubly incompressible lamination, then $\lambda$ can be realized as the pleating locus of a pleated surface homotopic to the inclusion of $T$ in $N$ (\cite{lec}, Theorem 5.1).
\subsubsection{Ends of hyperbolic manifolds}\label{ends}
The ends of $N$ are in one-to-one correspondence with the components of $\partial C-P,$ where $C$ is a relative compact core and $P$ is the intersection of $C$ with the noncompact components of $N_{\thin(\epsilon)}$ (for a precise definition of ends see \cite{kap}, Section 4.23).  Any hyperbolic $3$-manifold $N$ with finitely generated fundamental group has finitely many ends.  An end is \emph{geometrically finite} if it has a neighborhood $U$ which does not intersect $C(N)$ and is \emph{geometrically infinite} otherwise.  By the Tameness Theorem (\cite{ago}, \cite{cal-gab}), we can choose a relative compact core $C$ such that $N_\epsilon^0 - \inte(C)$ is homeomorphic to $\partial C - P \times [0, \infty),$ where $\inte(C)$ is the interior of $C$.  Suppose that $E$ is a component of $N_\epsilon^0 - \inte(C)$ homeomorphic to $T \times [0, \infty).$  If $E$ is geometrically infinite, then there exists a sequence $\alpha_i$ of closed geodesics, homotopic in $E$ to simple closed curves on $T$ that leave every compact set of $E$.  Fix a hyperbolic surface $T'$ and a homeomorphism $T' \rightarrow T$.  If $\alpha_i'$ is the geodesic on $T'$ mapping to $\alpha_i,$ then $\alpha_i'/l_{T'}(\alpha_i')$ converges in $\mathcal{ML}(T')$ to a measured lamination $(\lambda, \mu)$ such that its support $\lambda$ is independent of the sequence $\{\alpha_i\}$ (\cite{bon1}, \cite{can2}).  In this situation $\lambda$ is called the ending lamination for the end $E$.  It is minimal and is the support of a filling doubly incompressible lamination (\cite{can2} Corollary 10.2).

\subsubsection{Cannon-Thurston maps}\label{ctmaps}
Let $X$ and $Y$ be Gromov hyperbolic metric spaces and $i: X \rightarrow Y$ an embedding.  A \emph{Cannon-Thurston} map for $i$ is a continuous extension $\hat i: \hat X = X \cup \partial X \rightarrow \hat Y=Y \cup \partial Y$ where $\partial X$ and $\partial Y$ are the geodesic boundaries of $X$ and $Y,$ respectively.  By continuity, if $\hat i$ exists, it is unique.  If $i$ is a quasi-isometric embedding, then the existence of $\hat i$ is immediate since two geodesics that are within a bounded distance of each other in $X$ map to two quasi-geodesics that are within a bounded distance of each other in $Y$.  Cannon and Thurston (\cite{can-thu}) showed the existence of such maps when $Y$ is the Cayley graph of the fundamental group a closed hyperbolic $3$-manifold fibering over the circle, and $X$ is the subgraph associated to the fiber subgroup.

Suppose $\rho$ is a discrete and faithful representation of $G$ into $\textup{PSL}(2, \mathbb{C})$ and $\tau_\rho: C_S(G) \rightarrow \H^3$ is an orbit map (see Section \ref{ssreps} for more details on this map).  The existence of Cannon-Thurston maps for $\tau_\rho$ and the characterization of points that are not mapped injectively by such maps is a well-studied problem.  The results we will use in this paper are due to Floyd (\cite{flo}) for the case when $\rho$ is geometrically finite, and Mj (\cite{mj}) for the case when $\rho$ is geometrically infinite without parabolics.

\begin{theorem}[Floyd]
Let $\rho$ be a geometrically finite representation of $G$.  Then, $\tau_\rho: C_S(G) \rightarrow \H^3$ extends continuously to $\bar \tau_\rho: \overline{C_S(G)} \rightarrow \overline{ \H^3}$.  Moreover, $\bar \tau_\rho$ is $2:1$ onto parabolic points of rank one and injective elsewhere.
\end{theorem}

In order to state Mj's characterization of which points map noninjectively, we need a way to identify endpoints of leaves of ending laminations with points in $\partial C_S(G)$.  Suppose that $\rho$ is a purely hyperbolic geometrically infinite representation of $G$ into $\pslc$.  Let $E$ be a geometrically infinite end of $N_\rho$.  Recall from Section \ref{ends} that we can pick a standard compact core $C$ of $N_\rho$ such that the component of $N_\rho - C$ corresponding to $E$ is homeomorphic to $T \times (0, \infty)$ where $T$ is a boundary component of $M$.  Morever, there is a well-defined ending lamination $\lambda$ on $T$.  Then, $\lambda$ is doubly incompressible (\cite{can2}, Corollary 10.2) if $N_\rho$ has compressible boundary and $\lambda$ is filling if $N_\rho$ has incompressible boundary.  In either case, if $\rho'$ is any convex cocompact representation such that $N_\rho'$ is homeomorphic to $N_\rho,$ then $\lambda$ is realizable by a pleated surface, $p': T \rightarrow N_{\rho'}$ homotopic to the inclusion of $T$ in $N_{\rho'}$.  If $l$ is a leaf of $\lambda,$ then $p'(l)$ is a geodesic and its lift in $\H^3$ has two well-defined endpoints in $\Lambda(\rho'(G))$.  Since $\rho'$ is convex cocompact, $\tau_\rho'$ is a quasi-isometric embedding.  Hence it has a continuous extension $\overline{\tau_\rho'}$ that restricts to a homeomorphism from $\partial C_S(G)$ to $\Lambda(\rho'(G))$.  Under this homeomorphism, the endpoints of $l$ can be identified with two distinct points in $\partial C_S(G)$.

This identification is independent of the choice of pleated surface, since the image of the pleating locus is independent of the choice of pleated surface.  This identification is also independent of the choice of $\rho'$ for if $\rho''$ is another choice of convex cocompact representation with $N_{\rho''}$ homeomorphic to $N_{\rho'},$ then $\partial \tau_{\rho''} \circ \partial \tau_\rho{'}|_{\Lambda(\rho'(G))}^{-1}$ is a homeomorphism from $\Lambda(\rho'(G))$ to $\Lambda(\rho''(G))$ that sends the attracting fixed point an element $\rho'(g)$ to the attracting fixed point of $\rho''(g)$.  If $x$ is an endpoint of a leaf of $p'(l)$, we can find $x_i \rightarrow x$ such that $x_i$ is the attracting fixed point of $\rho'(g_i).$  Then, the attracting fixed points of $\rho''(g_i)$ approach an endpoint of $p''(l)$.  We are now ready to state Mj's result.

\begin{theorem}[Mj]
Let $\rho$ be a purely hyperbolic representation with one geometrically infinite end.  Then, $\tau_\rho: C_S(G) \rightarrow \H^3$ extends continuously to $\bar \tau_\rho: \overline{C_S(G)} \rightarrow \overline{ \H^3}$.  Moreover, $\bar \tau_\rho(a) = \bar \tau_\rho(b)$ for $a \neq b$ in $\partial C_S(G)$ if and only if $a$ and $b$ are either end-points of a leaf of an ending lamination or boundary points of a complementary ideal polygon.
\end{theorem}

\section{Separable-stable representations}\label{ssreps}
In this section we describe the set of separable-stable representations and show that it is a domain of discontinuity.  The definition and proof closely follow Minsky's argument in the case that $M$ is a handlebody in \cite{min}.  Much of the content in this section applies more generally than in the situation of a compression body and the general situation is carefully described in \cite{thesis}.  We omit the proofs of some general lemmas and provide references where one can find detailed proofs.

\begin{definition}
If $M$ is a compression body that is not the connect sum of two trivial $I$-bundles over closed surfaces, an element $g$ in $\pi_1(M)$ is \emph{separable} if it corresponds to a loop in $M$ that can be freely homotoped to miss an essential disc.
If $M$ is the connect sum of two trivial $I$-bundles over closed surfaces, an element $g$ in $\pi_1(M)$ is \emph{separable} if it corresponds to a loop in $M$ that can be freely homotoped to miss an essential annulus contained in one of the two trivial $I$-bundles.
\end{definition}

Let $G$ denote the fundamental group of $M$.  For each $g$ in $G$, let $g_-$ and $g_+$ denote the repelling and attracting fixed points of $g$ acting on $\partial C_S(G)$.  Let $\mathcal{L}_{S}(g)$ denote the set of geodesics connecting $g_-$ and $g_+$ and $\mathcal{S}_S$ denote the set of geodesics $l$ in $C_S(G)$ such that $l$ is contained in $\mathcal{L}_{S}(g)$ for some separable element $g$.

Given a representation $\rho: G \rightarrow \pslc)$ and a basepoint $x$ in $\H^3$, there exists a unique $\rho$-equivariant map $\tau_{\rho, x}: C_S(G) \rightarrow \H^3$ taking  the identity to $x$ and edges to geodesic segments.

\begin{definition}
A representation $\rho: G \rightarrow \pslc$ is called \emph{$(K, A)$-separable-stable} if there exists a basepoint $x$ in $\H^3$ such that $\tau_{\rho, x}$ takes all geodesics in $\mathcal{S}_S$ to $(K, A)$-quasi-geodesics.  A representation $\rho$ is \emph{separable-stable} if there exists $(K, A)$ such that $\rho$ is $(K, A)$-separable-stable.
\end{definition}

Separable-stability does not depend on the choice of $x$ or $S$ and is invariant under conjugation (see \cite{thesis} Lemma III.2).  Let $\mathcal{SS}(M)$ denote the set of separable-stable representations in $\X(M).$

The goal of this section is to show that $\mathcal{SS}(M)$ is a domain of discontinuity.

\begin{proposition} \label{compdisc}
$\mathcal{SS}(M)$ is a domain of discontinuity for the action of $\out(\pi_1(M))$ containing the interior of $AH(M)$.
\end{proposition}

\begin{proof}
We start by showing that the set of separable-stable representations is open.  Openness follows immediately from the following two lemmas (for proofs see \cite{thesis} Lemmas II.10 and III.5).  The first is a characterization of quasi-geodesics in terms of a nesting condition.
\begin{lemma} [Minsky] \label{qg condition}
Let $G$ be a finitely-generated hyperbolic group acting by isometries on $\H^3$.  Given $(K, A)$ there exists $c > 0$ and $i \in \mathbb{N}$ such that if $L'=\tau(L)$ is a $(K, A)$-quasi-geodesic, then $P_{j,i}$ separates $P_{(j+1),i}$ and $P_{(j-1),i}$ and $d(P_{j,i}, P_{(j+1),i})>c$. Conversely, given $c>0$ and $i \in \mathbb{N}$ there exists $(K', A')$ such that if $L'=\tau(L)$ has the property that $P_{j,i}$ separates $P_{(j+1),i}$ and $P_{(j-1),i}$ and $d(P_{j,i}, P_{(j+1),i})>c$ then $L'$ is a $(K', A')$-quasi-geodesic.
\end{lemma}

Then, openness of $\mathcal{SS}(M)$ follows from the following lemma.
\begin{lemma} \label{qgonopensets}
For any separable-stable representation $[\rho_0]$ in $\mathcal{X}(M)$, there exists a neighborhood $U_{[\rho_0]}$ of $[\rho_0]$ and constants $c'>0$, $i' \in \mathbb{N}$ such that for any $[\sigma]$ in $U_{[\rho_0]}$ and any geodesic $l$ in $\mathcal{S}_S$, the hyperplanes, $P_{j,i'}$ corresponding to $\tau_{\sigma, x}(l)$ have the property that $P_{j,i'}$ separates $P_{(j+1),i'}$ and $P_{(j-1),i'}$ and $d(P_{j,i'}, P_{(j+1),i'})>c'$.
\end{lemma}

To see that $\mathcal{SS}(M)$ is $\out(\pi_1(M))$-invariant it suffices to show that any automorphism $f$ of $\pi_1(M)$ preserves the set of separable elements.  Then, the isometry from $C_{f(S)}(G) \rightarrow C_S(G)$ that is the identity map on vertices will send the elements of $\mathcal{S}_{f(S)}$ to $\mathcal{S}_S$.  Since the image of $\tau_{\rho \circ f^{-1}, x}: C_{f(S)}(G) \rightarrow \H^3$ coincides with the image of $\tau_{\rho,x}: C_S(G) \rightarrow \H^3$, if $\rho$ sends elements in $\mathcal{S}_S$ to $(K, A)$-quasi-geodesics, then $\tau_{\rho \circ f^{-1}, x}$ sends elements in $\mathcal{S}_{f(S)}$ to $(K,A)$-quasi-geodesics.  Since separable-stability is independent of the choice of generators of $G$, we have that $\rho \circ f^{-1}$ will also be separable-stable.

If $M$ is a compression body that is not the connect sum of two trivial $I$-bundles, then an element $g$ is separable if and only if $g$ lies in a proper factor of a decomposition of $M$ into a free product.  Indeed, if $g$ corresponds to a curve that is freely homotopic to a curve missing an essential separating disc corresponding to the splitting $H*K$, then $g$ lies in either $H$ or $K$, up conjugation.  If $g$ corresponds to a curve that is freely homotopic to a curve missing an essential nonseparating disc, then there is a splitting $H*_{\{1\}}$ such that $g$ lies in $H$, up to conjugation.  Since $H*_{\{1\}} \cong H*\mathbb{Z}$, we have that $g$ lies in a proper factor of a free decomposition.  The converse follows from Lemma \ref{mer}.  This implies that separability is preserved under automorphisms.

For the remaining case, suppose that $M$ is the connect sum of $S_1 \times I$ and $S_2 \times I$ where $S_i$ is a closed surface of genus at least two.  First, we want to see that $g$ is separable if and only if there is a decomposition of $\pi_1(M)$ as $\pi_1(M) \cong (H*_{<c>} K) * L$ or $\pi_1(M) \cong (K*_{<c>}) * L$ such that $g$ lies in $K*L$ satisfying the following
\begin{enumalph}
\item $L \cong \pi_1(S_i)$
\item $H*_{<c>} K \cong \pi_1(S_j)$ in the first type of decomposition or $(K*_{<c>}) \cong \pi_1(S_j)$ in the second type of decomposition where $i \neq j$
\item $c$ is freely homotopic to a simple closed curve on $S_j$
\end{enumalph}
If $g$ is separable, then it misses an essential annulus $A$ in $S_j \times I$.  If $c$ is the core curve of $A$, then $\pi_1(S_j)$ decomposes as $H*_{<c>}K$ or $K*_{<c>}$.  Then $\pi_1(M)$ decomposes as $\pi_1(M) = H*_{<c>}K * \pi_1(S_i)$ or $\pi_1(M) = \pi_1(S_i) * K*_{<c>}$ such that $g$ lies in $K * \pi_1(S_i)$.  Conversely, if $g$ is an element in $\pi_1(M)$ that lies in $K*L$ for a decomposition of one of the above types, then $g$ misses the essential annulus $c \times I$.

Now, it suffices to show that such a decomposition is preserved under automorphisms.  If $\phi$ is an automorphism of $\pi_1(M)$, then $\phi((H*_{<c>} K) * L)=(\phi(H)*_{<\phi(c)>} \phi(K)) * \phi(L)$.  By the Kurosh subgroup theorem (\cite{kur}), $\phi(H)*_{<\phi(c)>} \phi(K)$ is conjugate to $\pi_1(S_j)$ and $\phi(L)$ is conjugate to $\pi_1(S_i)$.  Up to composition with an inner automorphism and potentially switching the factors, $\phi (\pi_1(S_j))=\pi_1(S_j)$ and $\phi(L)=\pi_1(S_i)$ (\cite{can-mcc} Lemma 9.1.2).  Since homotopy equivalence of closed surfaces are homotopic to homeomorphisms, $\phi(c)$ is a simple closed curve on $S_j$.

That $\mathcal{SS}(M)$ contains the interior of $AH(M),$ follows from a result of Sullivan (\cite{sul}) that the interior of $AH(M)$ consists of convex cocompact representations and hence have orbit maps that are quasi-isometric embeddings.

It remains to show that the action of $\out(\pi_1(M))$ on $\mathcal{SS}(M)$ is properly discontinuous.  The idea is that separable-stability will imply that translation length of a separable in the Cayley graph is coarsely the same as translation length of the corresponding isometry in $\H^3$.  To show proper discontinuity of the action it will suffice to show that only finitely many automorphisms, up to conjugation, can change the translation length of all separable elements in the Cayley graph by a bounded amount.

Let $l_\rho(g)$ denote the translation length of $\rho(g)$ in $\H^3,$ and let $ || g ||$ denote the minimum translation length of $g$ in $C_S(G)$.  The following lemma (see \cite{thesis} Lemma III.7) states that for each compact set in $\mathcal{SS}(M)$, the translation length of a separable element in $\H^3$ and in $C_S(G)$ is coarsely the same.  For each $(K,A)$-separable-stable representation, since the orbit of a separable element is a $(K, A)$-quasi-geodesic the translation length in $\H^3$ and in $C_S(G)$ are coarsely the same.  By Lemma \ref{qgonopensets}, in a compact subset of $\mathcal{SS}(M)$ one can choose uniform constants $(K,A)$.

\begin{lemma}Let $C$ be a compact subset of $\mathcal{SS}(M)$.  There exists $r, R>0$ such that
$$
r \leq \frac{l_\rho(g)}{||g||}\leq R
$$
for all $g$ separable and all representations $[\rho]$ in $C$.
\end{lemma}

Now suppose that $[f]$ is an element in $\out(G)$ such that $f(C) \cap C \neq \emptyset$.  Applying the above inequalities we have that for any $[\rho] $ in $C$ and any $w$ separable,

$$
||f^{-1}(w)|| \leq \frac{1}{r} l_\rho(f^{-1}(w))= \frac{1}{r} l_{\rho \circ f^{-1}}(w) \leq  \frac{R}{r}||w||.
$$

We will need that there are a sufficient number of separable elements.  In particular, there is a generating set of $\pi_1(M)$ such that each generator and any two fold product of generators is separable.  Indeed if $M$ is a large compression body, take a maximal decomposition of $G$ into a free product, $G=G_1 * \cdots * G_n$.  Let $X$ be the union of finite generating sets for each factor.  Since $n$ is at least three, any two fold product of distinct generators is separable.  If $M$ is a small compression body, let $X$ be the union of the standard generators of each closed surface group factor and a generator for the infinite cyclic factor if there is a handle.  In the case that $M$ is a connect sum of two trivial $I$-bundles over closed surfaces, it is clear that any two fold product of such generators is separable.  In the case that $M$ is a trivial $I$-bundle over a closed surface with a one handle, we only need to concern ourselves with products of the form $x_{i_1}x_{i_2}$ where $x_{i_1}$ is part of the generating set for the closed surface group and $x_{i_2}$ is the generator for the infinite cyclic factor.  Then, $x_{i_2}$ misses $D$ an essential disc.  Let $f$ be an automorphism of $G$ that is the identity on the surface group factor and maps $x_{i_2}$ to $x_{i_1}x_{i_2}$.  By the discussion in the proof of Lemma \ref{mer}, $f$ is realizable by a homeomorphism $f'$.  Hence $x_{i_1}x_{i_2}$ misses the essential disc $f'(D)$.
Let $\mathcal{W}=\{x_i, x_ix_j| x_i, x_j \in X\}$.  Then, by the following lemma (see \cite{min} Lemma 3.4 or \cite{can-survey} Proposition 2.3), the action is properly discontinuous.

\begin{lemma}
For any $N > 0$, the set $$\mathcal{A}=\{[f] \in \out(G) | \text{    } ||f(w)|| \leq N ||w|| | \text{ for all $w \in \mathcal{W}$} \}$$ is finite.
\end{lemma}

\end{proof}

\section{Separable-stable points on $\partial AH(M)$} \label{boundarypoints}
The goal of this section is to prove the existence of separable-stable points on $\partial AH(M)$ (Proposition \ref{pspoints}).  Together with Proposition \ref{compdisc} these examples will complete the proof of Theorem \ref{compbody}.  These points will correspond to pinching either a Masur domain curve or a Masur domain lamination on the exterior boundary of $M.$  Using Lemma \ref{compact}, it suffices to show that all separable geodesics lie in a compact set.  Roughly speaking, if this were not the case, then one could find a sequence of fixed points of separable elements in $\partial C_S(G)$ approaching an endpoint of an end invariant.  We use the Whitehead graph to form a dichotomy between separable elements and Masur domain laminations to show that such a situation is impossible.

\subsection{The Whitehead graph for a compression body}\label{whitehead}
In this section we define the Whitehead graph for a closed geodesic or Masur domain lamination with respect to a fixed system of meridians $\alpha$ on $M$.  The generalization of Whitehead graphs to compression bodies was developed in Otal's Th\`ese d'Etat (\cite{ota}).

\subsubsection{The handlebody case}
We will start by describing Whitehead's original construction in the case when $M$ is a handlebody (\cite{whi1}, \cite{whi2}, see \cite{sta}).  As this is discussed in detail in \cite{min}, we will sketch this case and discuss the case of compression bodies that are not handlebodies in detail.  For a fixed free symmetric generating set $X= \{x_1, \dots, x_n, x_1^{-1}, \ldots, x_n^{-1}\}$ of $\pi_1(M)$ and a word $w=w_1\cdots w_k$ in $\pi_1(M)$, the Whitehead graph of $w$ with respect to $X$ is the graph with $2n$ vertices $x_1, x^{-1}, \ldots, x_n, x_n^{-1}$ and an edge from $x$ to $y^{-1}$ for each string $xy$ in $w$ or any cyclic permutation of $w$.  In this situation Whitehead (\cite{whi1}) proved the following.

\begin{lemma} (Whitehead)
Let $g$ be a cyclically reduced word.  If the Whitehead graph of $g$ with respect to $X$ is connected and has no cutpoint, then $g$ is not primitive.
\end{lemma}
A word $w_1 \cdots w_k$ is cyclically reduced if it is reduced and satisfies $w_1 \neq w_k^{-1}$.  A cutpoint
is a vertex whose complement is disconnected.  A primitive element in a free group is an element that lies in a free generating set for the group.  In particular, a primitive element is separable.

Otal extended Whitehead's condition to laminations on the boundary of the handlebody as follows.  If $X$ is a free generating set for $\pi_1(M)$, then it is dual to a system of properly embedded essential discs on $M$ whose complement is a 3-ball.  If $D=\{D_1, \ldots, D_n\}$ is such a system of disks, then Otal calls a lamination $\lambda$ in \emph{tight position} with respect to $D$ if there are no waves on $D$ disjoint from $\lambda$.  A \emph{wave} is an arc $k$ properly embedded in $\partial M - \partial D$ such that $k$ is homotopic in $M$ but not in $\partial M$ into $\partial D$.  Cutting $\partial M$ along $D$, produces a planar surface with $2n$ boundary components $D_1^+, D_1^-, D_2^+, D_2^-, \ldots, D_n^+, D_n^-$.  The vertices of the Whitehead graph are in one-to-one correspondence with these boundary components.  There is an edge between two vertices if there is an arc of $\lambda - \partial D$ connecting the corresponding boundary components.  Otal proved that for laminations in the Masur domain of $M$ that are in tight position with respect to $D$, the Whitehead graph is connected and has no cutpoints (\cite{ota} Proposition 3.10, see \cite{min} Lemma 4.5).  Moreover, if $\lambda$ is a Masur domain lamination, then there always exists a system of discs $D$ such that $\lambda$ is in tight position with respect to $D$.

\subsubsection{Compression bodies that are not handlebodies}
Here we will discuss the Whitehead graph of compression bodies that are not handlebodies.
Fix $\sigma: G
\rightarrow \pslc$ a convex cocompact representation of
$G$ such that $N_\sigma=\H^3/\sigma(G)$ is homeomorphic to the interior of $M$.

\begin{definition} A \emph{system of meridians} is a collection
$\alpha=\{\alpha_1, \ldots, \alpha_n\}$ of disjoint, pairwise nonisotopic, simple
closed curves on $\partial_{ext}M$ that bound discs $D=\{D_1, \ldots, D_n\}$ in $M$ such that
$M-\mathcal{N}(D)$ consists of a collection of trivial $I$-bundles over closed surfaces, where $\mathcal{N}(D)$ is a regular neighborhood of $D.$
\end{definition}

Since $\sigma$ is convex cocompact, $\overline{N_\sigma}= N_\sigma \cup \partial_C N_\sigma$ is homeomorphic to $M$.  We will often identify $\partial M$ with $\partial_C N_\sigma$.  Let $\alpha$ be a system of meridians bounding the discs
$D=D_1 \cup \cdots \cup D_n$.  Let $\Sigma_1 \times I, \ldots, \Sigma_k \times I$ be the components of $\overline{N_\sigma}-\mathcal{N}(D).$  Let $\mu$ be a subset of $\Lambda(\sigma(G)) \times \Lambda(\sigma(G))$ that is $\sigma(G)$-invariant and also invariant under switching the two factors.  Most of the time $\mu$ will be one of the two following sets:
\begin{itemize}
\item If $\gamma$ is a closed geodesic in $N_\sigma$, let $\mu_\gamma$ be all pairs of endpoints of the lifts of $\gamma$.
\item If $\lambda$ is a Masur domain lamination, then it is realizable by a pleated surface $h: S \rightarrow N_\sigma$ homotopic to the inclusion map.  Let $\mu_\lambda$ be all pairs of endpoints of lifts of leaves in $h(\lambda).$
\end{itemize}

Recall from Section \ref{psurf} that $h(\lambda)$ is independent of the choice of pleated surface.  $\Gamma_\alpha(\mu)$, the Whitehead graph of $\mu$ with respect to $\alpha$, is a collection of not necessarily connected graphs, $\Gamma_\alpha(\mu)^{\Sigma_1}, \ldots, \Gamma_\alpha(\mu)^{\Sigma_k},$ where the elements in the collection are in one-to-one
correspondence with the components of $\overline{N_\sigma}-D$.  In $\overline{N_\sigma}-\mathcal{N}(D)$, there are two copies $D_i^+$ and $D_i^-$ of each $D_i$ in $D$.  Given a component $\Sigma \times I$
of $\overline{N_\sigma}-\mathcal{N}(D)$, the vertices in the corresponding graph $\Gamma_\alpha(\mu)^\Sigma$
are in one-to-one correspondence with the components of $D_i^+$ and/or $D_i^-$ of $D$ in the frontier
of $\Sigma \times I$.  Fix a Jordan curve $C$ in $\Lambda(\sigma(G))$ that is invariant under a conjugate
of $\pi_1(\Sigma)$, which we will continue to denote $\pi_1(\Sigma)$.  To avoid superscripts, we will abuse notation and let $D_1,
\ldots, D_m$ denote the vertices of this component.  Let $F$ denote the boundary component of
$\Sigma \times I$ coming from $\partial_{ext}M$.  Fix a lift $\widetilde{\partial D_i}$ of each $\partial D_i$ in $\partial \H^3$ such that
$\widetilde{\partial D_i}$ lies in the component of the preimage of $F$
containing $C$ on its boundary.  Let $U_i$ be the open set in $\partial \H^3$, bounded by $\widetilde{\partial D_i}$ not containing $C$.  The edges from $D_i$ to $D_j$ will be in
one-to-one correspondence with elements $g$ in $\pi_1(\Sigma)$ such that $\mu \cap (U_i \times g U_j)$ is nonempty.  We will denote such an edge $(U_i, g U_j)$.  Notice
that although these edges are directed, for each edge from $U_i$ to $U_j$ labeled
$g$, there is an edge from $U_j$ to $U_i$ labeled $g^{-1}$.

\begin{definition} [Otal] A connected component of $\Gamma_\alpha(\mu)^\Sigma$ is \emph{strongly connected} if there exists a cycle that represents a nontrivial element of
$\pi_1(\Sigma)$.  A connected component of $\Gamma_\alpha(\mu)^\Sigma$ has a \emph{strong cutpoint} if we can express the graph as the
union of two graphs $G_1$ and $G_2$ that intersect in a single vertex such that
either $G_1$ or $G_2$ is not strongly connected. \end{definition}

We take the convention that a cycle $$(U_{i_1}, g_1U_{i_2}), (U_{i_2},g_2U_{i_3}),
\ldots, (U_{i_k},g_k U_{i_1})$$ corresponds to the group element $g_1\cdots g_k$.  We
made two choices when defining the Whitehead graph, the lifts $U_i$ of $D_i$ and
the Jordan curve $C$.  Suppose that we pick a different set of lifts $U_1',
\ldots, U_m'$ of $D_1, \ldots, D_m$.  Then $U_i'=h_iU_i$ for some $h_i$ in
$\pi_1(\Sigma)$.  There is an edge $(U_{i_1}, gU_{i_2})$ in the original graph if and only if there is an edge $(U_{i_1}', h_{i_1}gh_{i_2}^{-1}U_{i_2}')$ in the new graph.  In particular, there is a cycle $(U_{i_1}, g_1U_{i_2}), \ldots, (U_{i_k},g_k
U_{i_1})$ in the original graph if and only if there is a cycle
$$(U_{i_1}',
h_{i_1}g_1h_{i_2}^{-1}U_{i_2}'), \ldots, (U_{i_k}', h_{i_k}g_k
h_{i_1}^{-1}U_{i_1}')$$ in the new graph.  Since $g_1 \cdots g_k$ is
nontrivial if and only if $h_{i_1}g_1\cdots g_k h_{i_1}^{-1}$ is nontrivial, the
above definitions do not depend on the choice of lifts $U_i$.

Suppose we choose a different Jordan curve $C'$.  Then there exists an element $a$
in $G$ such that $C'=aC$.  The lifts $aU_i$ of $D_i$ will lie in the appropriate
component of the preimage of $F$ in $\partial H^3$, namely the one containing $C'$
on its boundary.  Since $\mu$ is $\sigma(G)$-invariant there is an edge between $U_i $ and $gU_j$ if and only if there is an edge between $aU_i$ and $aga^{-1}aU_j$; in particular, any edge
labeled $g$ in the original graph is now an edge in the new graph labeled
$aga^{-1}$.  So the above definitions do not depend on the choice of Jordan curve
$C$.

\subsubsection{Topological Interpretation} \label{topinterpret}
In the case when $\mu=\mu_\lambda$ for a Masur domain lamination $\lambda$, Otal describes a topological interpretation of the Whitehead graph in terms of the exterior boundary.  As in the handlebody case, there is a notion of tight position.
\begin{definition} A measured lamination $\lambda$ on $\partial_{ext}M$ is in \emph{tight position
relative to $\alpha$} a system of meridians if there does not exist a wave $k$ disjoint from $\lambda$ properly embedded in
$S-\alpha$, where a \emph{wave} is an arc satisfying the following.
 \begin{itemize}
 \item the interior of $k$ is disjoint from $\alpha$.
 \item $k$ can be homotoped in $N_\sigma$ but not in $\partial M$ relative to its endpoints to an arc contained in some $\alpha_i$.
 \end{itemize}
\end{definition}

Observe that if $\lambda$ is in tight position, then there are no waves in $\lambda$.  If $\lambda$ is a Masur domain lamination then there exists a system of meridians
$\alpha$ such that $\lambda$ is in tight position with respect to $\alpha$ (\cite{mas}, Section 3 for handlebodies, \cite{ota} Theorem 1.3 for general compression bodies).
Such a system is obtained by minimizing the intersection number with $\lambda$.
Assume
that $\lambda$ is in tight position with respect to $\alpha$.  We will start by defining a related collection of graphs denoted $\Gamma_\alpha'(\lambda)$.  Consider
the collection of surfaces with boundary obtained by cutting $\partial_{ext}M$ along $\alpha$.
For each $\alpha_i$ in $\alpha$, there are two boundary components $\alpha_i^+$
and $\alpha_i^-$ in the new collection of surfaces with boundary. For each
component $F$ of $S-\alpha$, we will define a graph of $\Gamma'_\alpha(\lambda)^F$ as follows.
The vertices will be in one to one correspondence with the copies of $\alpha_i^+$ and/or $\alpha_i^-$ in
the frontier of $F$.  We will abuse notation and relabel the boundary components $\alpha_i$ to avoid superscripts.  The edges from the vertex $\alpha_i$ to the vertex $\alpha_j$ are
in one-to-one correspondence with the isotopy classes of arcs on $\partial_{ext}M$ in $\lambda$
connecting $\alpha_i$ and $\alpha_j$.

There is a natural surjective map from $\Gamma'_\alpha(\lambda) \rightarrow
\Gamma_\alpha(\mu_\lambda)$ defined as follows.  Take the obvious map on the
vertices.  Suppose that $[k]$ is an edge connecting $\alpha_i$ and $\alpha_j$.
Let $D_i$ and $D_j$ denote the corresponding vertices in
$\Gamma_\alpha(\mu_\lambda)$ and let $U_i$ and $U_j$ be the fixed lifts of $D_i$ and $D_j$, respectively.  Take the lift $\widetilde k$ of $k$ intersecting $U_i$.  Since $\lambda$ is in tight position with respect to $\alpha$, we have that $\widetilde k$ will have one endpoint in $U_i$ and the other in $gU_j$ for some $g$ in $\pi_1(\Sigma)$.  Map the edge
$[k]$ in $\Gamma'_\alpha (\lambda)$ to the edge $(U_i, gU_j)$ in $\Gamma_\alpha
(\mu_\lambda)$.  To see that $(U_i, gU_j)$ is an edge in
$\Gamma_\alpha(\lambda)$ we will need the following two facts.

\begin{itemize}
\item Any lift of a leaf $l$ of $\lambda$ has two well-defined endpoints $x_1$ and $x_2$ in $\Lambda(\sigma(G)).$
\item $(x_1,x_2)$ are the endpoints of $h(l),$ where $h: S \rightarrow N_\sigma$ is a pleated surface realizing $\lambda.$
\end{itemize}

To see a proof of the first fact see Lemma 1 in \cite{kle-sou}.  The second fact is clear since it is true for simple closed curves in the Masur domain and we can approximate $\lambda$ by such curves.

Now, if we consider the leaf $\widetilde l$ of $\widetilde \lambda$ containing the arc $\widetilde k$ its endpoints must be contained in $U_i$ and $gU_j$ by tightness.

To see that the map is surjective, given an edge $(U_i, gU_j)$ in
$\Gamma_\alpha(\mu_\lambda)$, there exists a leaf $\widetilde l$ with endpoints in $U_i$
and $gU_j$.  This will give an arc between $\alpha_i$ and $\alpha_j$.   Two edges
$[k]$ and $[k']$ in $\Gamma'_\alpha(\lambda)$ are identified in $\Gamma_\alpha(\mu_\lambda)$ if and only if they are homotopic in
$N_\sigma$ (see Figure \ref{topinterpic}). Hence edges in the Whitehead graph between $D_i$ and $D_j$ correspond to
homotopy classes of arcs of $\lambda$ joining $\partial D_i$ and $\partial D_j$.

\begin{figure}
\begin{center}
 \includegraphics[width=60mm]{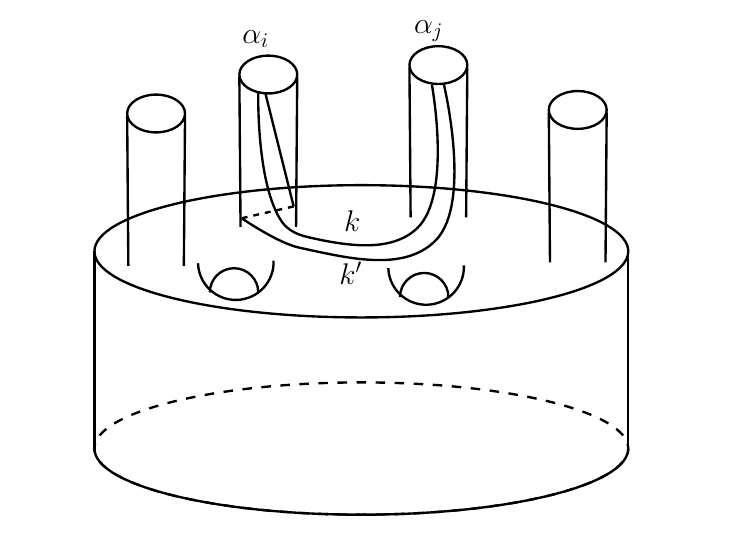}
\caption[Edges that get identified]{The edges $[k]$ and $[k']$ in $\Gamma'_\alpha(\mu)$ get identified in $\Gamma_\alpha(\mu)$.}
\label{topinterpic}
\end{center}
\end{figure}

\subsubsection{Whitehead graphs of separable curves and Masur domain laminations}
In this section we give Otal's generalization of Whitehead's lemma, namely that for a separable element $g$ there exists a connected component of $\Gamma_\alpha(\mu_g)$ that is either not strongly connected or has a strong cutpoint.  On the other hand, Otal also showed that the Whitehead graph of a Masur domain curve in tight position with respect to $\alpha$ is strongly connected and without a strong cutpoint.  We will use this dichotomy in Section \ref{examples}.

\begin{proposition} [Otal, Proposition A.3]  \label{primnsc} Let $M$ be a nontrivial compression body that is neither the connect sum of two trivial $I$-bundles over closed surfaces nor a handlebody.  If $g$ is a separable element of $G=\pi_1(M)$, then some connected component of $\Gamma_\alpha(\mu_g)$ is not strongly connected or has a
strong cutpoint.
\end{proposition}

\begin{proof} Let $\gamma$ be the geodesic representative of $\sigma(g)$ in $N_\sigma$.  Since $M$ is not the connect sum of two trivial $I$-bundles, by Lemma \ref{mer}, there exists an essential disc $\Delta$ disjoint from $\gamma$.  As $\Delta$ is only well-defined up to
isotopy we will often abuse notation and use $\Delta$ to refer to different representatives of the isotopy class of $\Delta$.

First consider the case where $\Delta$ does not intersect $D$, up to isotopy.
Then $\Delta$ is isotopic (rel boundary) into the boundary of some component
$\Sigma \times I$ of $\overline{N_\sigma}- \mathcal{N}(D)$.  Consider the sets $\mathcal{A}=\{D_i | D_i
\subset \Delta\}$ and $\mathcal{A}^c= \{D_i | D_i
\not\in \mathcal{A}\}$.
If $\mathcal{A}$ is empty, then $\Delta$ is isotopic to some $D_i$.  This implies that $D_i$ is an isolated vertex in the Whitehead graph.  Moreover, any edge connecting $D_i$ to itself, is labeled with the trivial element as $\gamma$ does not intersect $D_i$.  In particular, the connected component containing $D_i$ is not strongly connected.  If $\mathcal{A}$ is nonempty, first observe that no vertex in $\mathcal{A}$ can be connected to a vertex in $\mathcal{A}^c$.  Let $C$ be any connected component of $\Gamma_\alpha(\mu_\gamma)^\Sigma$ containing a vertex in $\mathcal{A}$.  Then, $C$ is not strongly connected for if $(U_{i_1}, g_1 U_{i_2}), \ldots, (U_{i_k}, g_k U_{i_1})$ is a cycle in $C$ and if $\widetilde \Delta$ is the lift of $\Delta$ containing $U_{i_1}$, then $\widetilde \Delta$ also contains $g_1 \cdots g_kU_{i_1}$ since $\gamma$ does not intersect $\Delta$.  In particular, the cycle is trivial as the stabilizer of $\widetilde \Delta$ is trivial.

If $\Delta$ and $D$ intersect nontrivially, up to isotopy, $\Delta \cap D$
is a finite collection of disjoint properly embedded arcs.  Take $k_0$ an innermost arc in this collection, meaning that one of the discs $\Delta_0$ formed by $\partial \Delta$ and $k_0$ has interior disjoint from $D$.  Then, $\Delta_0$ lies in some component $\Sigma \times I$ of
$\overline{N_\sigma}-\mathcal{N}(D)$.  So $\partial \Delta_0$ intersects one of the $D_i$ in the frontier of $\Sigma
\times I$ nontrivially.  Let $D_0$ denote that disc.  Moreover, $\Delta_0$ is
isotopic relative to its boundary into the boundary of $\Sigma \times I$.  Let $C$ denote the connected component of $\Gamma_\alpha(\mu_\gamma)^\Sigma$ containing $D_0$.
Consider $\mathcal{B}$ the set of vertices in $C$ such that the corresponding discs $D_i$ in $\Sigma \times I$ are contained in $\Delta_0$.  Notice that if $\mathcal{B}$ is empty, then we can isotope $\Delta$ to remove $k_0$ from the intersection and repeat the procedure above.  Let $\mathcal{C}$
denote the set of vertices in $C$ such that the corresponding discs $D_i$ in $\Sigma \times I$ are disjoint from $\Delta_0$.  The only vertex not lying
in either set is $D_0$.  We claim that $D_0$ is a strong cutpoint of $C$ where the graph associated to the vertices in $\mathcal{B}$ is not strongly connected.  In particular, if $\mathcal{C}$ is empty, then $C$ is not strongly connected.

First, we want to show that no vertex in $\mathcal{B}$ is connected by an edge to a
vertex in $\mathcal{C}$.  Suppose that there is an edge $(U_b, gU_c)$ where $U_b$ is the fixed lift of a vertex in $\mathcal{B}$ and $U_c$ is the fixed lift of a vertex in $\mathcal{C},$ i.e., there is a lift $\widetilde \gamma$ of $\gamma$ such that one endpoint lies in $U_b$ and the other endpoint lies in $gU_c$.  This implies that $\gamma$ must intersect $\Delta_0$ nontrivially, which contradicts how we chose $\Delta_0$.

Secondly we want to show that the subgraph associated to $\mathcal{B}$ is not
strongly connected.  Suppose there is a cycle $(U_0, g_0U_1), \ldots, (U_k, g_kU_0)$ such that $g_0 \cdots g_k$ is nontrivial.  If $\widetilde \Delta_0$ is the lift of $\Delta_0$ containing $g_0U_1$, then $U_0$ and $g_1 \cdots g_kU_0$ intersect $\Delta_0$ (see Figure \ref{primnscpic2}).  This is impossible as it implies that either there is a nontrivial curve in $\Delta_0$ or $\Delta_0 \cap D_0$ consists of two connected components, contradicting how we chose $\Delta_0$.
\end{proof}

\begin{figure}
\begin{center}
 \includegraphics[width=60mm]{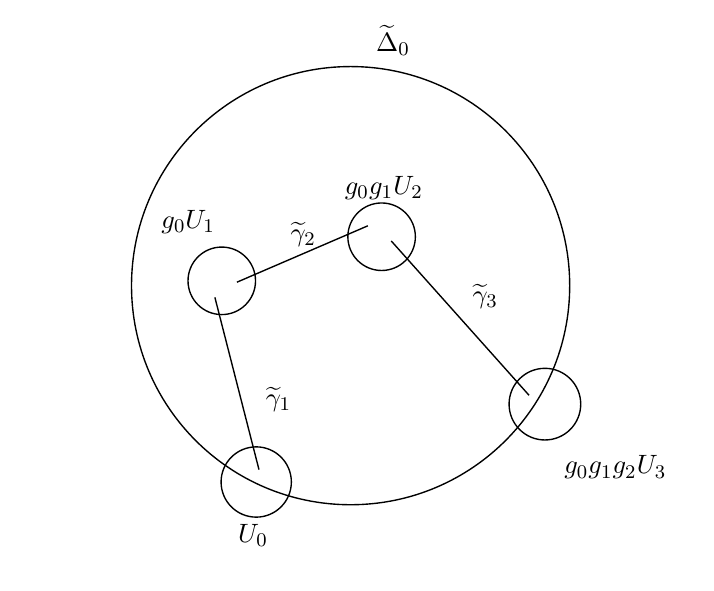}
\caption[No non-trivial cycles]{If there is a non-trivial cycle $(U_0, g_0U_1), \ldots, (U_k, g_kU_0)$, then $U_0$ and $g_0g_1g_2\cdots g_kU_0$ will both intersect $\widetilde \Delta_0$.}
\label{primnscpic2}
\end{center}
\end{figure}

\begin{proposition}[Otal, Propostion A.5] \label{masurdomain}
Let $\lambda$ be a measured lamination in the Masur domain
that is in tight position with respect to a system of meridians $\alpha$.  Then, each connected component of
$\Gamma_\alpha(\mu_\lambda)$ is strongly connected and without a strong cutpoint.
\end{proposition}

\begin{proof} We will use the topological interpretation of the Whitehead graph discussed in Section \ref{topinterpret}.
Suppose that a component $C$ of $\Gamma_\alpha(\mu_\lambda)^\Sigma$ is not strongly connected.  Let $D_1, \ldots, D_k$
be the components of $D$ that correspond to the vertices
in $C$.  Let $\mathcal{A}$ be the union of the $D_i$ and $\lambda \cap \Sigma$.  Take $\mathcal{N}(\mathcal{A})$ a regular neighborhood of $\mathcal{A}$.  The boundary of $\mathcal{N}(\mathcal{A})$ consists of simple closed curves that each bound a disc in $\Sigma$, as $C$ is not strongly connected.  One of these boundary components $b$ must bound a disc containing some $D_i$.  Then, $b$ is nontrivial on $\partial_{ext}M$ and so we have found a meridian that misses $\lambda$, a contradiction.

Suppose that $C$ has a strong cutpoint.  Let $F$ denote $\Sigma- \cup \inte(D_i).$  Let $D_0$ correspond to the strong cutpoint and let $G_1$ and
$G_2$ be the two graphs whose intersection is $D_0$ such that $G_1$ is not
strongly connected.  Let $\beta_1, \ldots, \beta_t$ be the merdians in $F$ corresponding to vertices of $G_1$. Let $\lambda' \subset \lambda \cap F$ consisting of arcs intersecting at least one $\beta_i$.  Let $\mathcal{N}$ denote a regular neighborhood of $\lambda' \cup (\cup \beta_i)$.  The boundary of $\mathcal{N}$ consists of closed curves $c_1, \ldots, c_l$ and arcs $a_1, \ldots, a_s$ with endpoints lying on $\alpha_0$.  Since $G_1$ is not strongly connected, each $c_i$ bounds a disc.  We claim that at least one of the arcs $a_i$ is a wave, i.e., an arc disjoint from $\lambda$, homotopic relative to its endpoints, in $M$ but not in $\partial_{ext}M$ into $\alpha_0$.  Indeed, any $a_i$ is disjoint from $\lambda$, by construction and homotopic in $M$ into $\alpha_0,$ since $G_1$ is not strongly connected.  For each arc $a_i$ choose an arc $b_i$ in $\alpha_0$ sharing the same endpoints as $b_i$ such that $a_i \cup b_i$ bounds a disc not containing $\alpha_0$.  At least one of the loops $c_1, \ldots, c_l, a_1 \cup b_1, \ldots, a_s \cup b_s$ contains some $\beta_i,$ since they form the boundary components of $\mathcal{N}$.  If $c_i$ contained some $\beta_i$, then $\beta_i$ would not be connected to $\alpha_0$, which contradicts how we chose $\beta_i$.  Therefore, some $a_k \cup b_k$ bounds a disc containing at least one $\beta_i$.  In particular, $a_k$ will not be homotopic in $\partial_{ext}M$ into $\alpha_0$.  So $a_k$ is a wave disjoint from $\lambda,$ a contradiction to the assumption that $\lambda$ is in tight position (see Figure \ref{masurdomainpic}).

\begin{figure}
\begin{center}
\includegraphics[width=80mm]{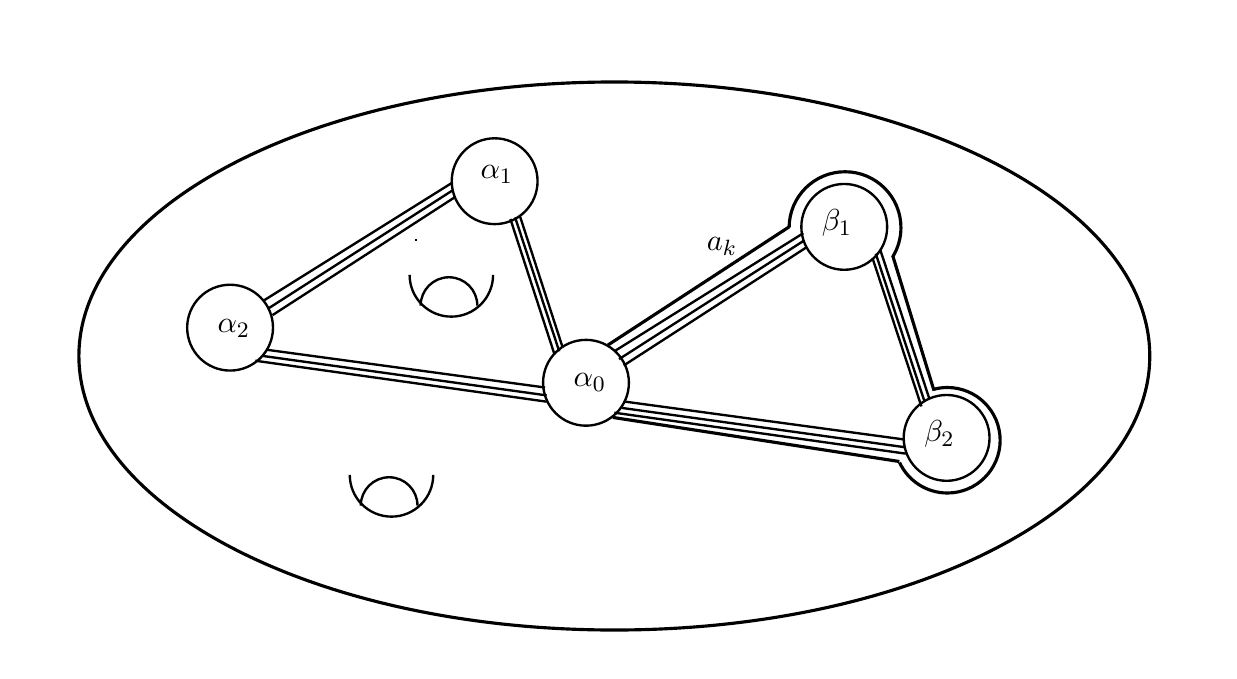}
\caption[A wave]{The arc $a_k$ is a wave disjoint from $\lambda.$}
\label{masurdomainpic}
\end{center}
\end{figure}

\end{proof}

\subsection{Examples of separable-stable points on $\partial AH(M)$} \label{examples}
In this section we prove Proposition \ref{pspoints}, which shows that two types of points on $\partial AH(M)$ are separable-stable; namely that a geometrically finite point with one cusp associated to a Masur domain curve is separable-stable and a purely hyperbolic geometrically infinite point with one geometrically infinite end corresponding to the exterior boundary component is separable-stable.  The case that a geometrically finite point with one cusp associated to a Masur domain curve for handlebodies is separable stable was proven by Minsky \cite{min}. The case that a purely hyperbolic geometrically infinite point is separable-stable for handlebodies was proven by Jeon-Kim in \cite{jeo-kim}.

\begin{lemma} \label{compact}
Let $\rho$ be a discrete and faithful representation of $\pi_1(M)$ into $\textup{PSL}(2, \mathbb{C}).$  Then $\rho$ is separable-stable if and only if
\begin{enumalph}
\item $\rho(g)$ is hyperbolic for any separable element $g$ and
\item there exists a compact subset $\Omega$ of $N_\rho = \H^3/\rho(\pi_1(M))$ such that the set of geodesics corresponding to separable elements of $\pi_1(M)$ is contained in $\Omega$.
\end{enumalph}
\end{lemma}

\begin{proof}
For the forward direction, suppose that $\rho$ is $(K, A)$-separable-stable.  Then, $\rho(g)$ must be hyperbolic for all separable-elements $g$ for if $\rho(g)$ were parabolic, then for any geodesic $l$ connecting the fixed points $g_+$ and $g_-$ on $\partial C_S(G)$, $\tau_{\rho, x}(l)$ would not be a quasi-geodesic.  For the second property, notice that elements of $\mathcal{S}$ stay within a bounded neighborhood of their corresponding geodesic axes in $\H^3$.  In particular, geodesics representing separable elements will stay in a bounded neighborhood of the image of the Cayley graph in $N_\rho$, which is a compact set.

Conversely, suppose that $\rho(g)$ is hyperbolic for all separable elements $g$ and that there exists a compact set $\Omega$ such that all separable-geodesics of $N_\rho$ are contained in $\Omega$.  Without loss of generality, since $N_\rho$ is tame (by \cite{ago} or \cite{cal-gab}),  we can assume that $\Omega$ is a compact core $C$ of $N_\rho$ containing the image of $C_S(G)/ \rho(G)$ in $N_\rho$.  This implies that $\tilde \Omega$, the preimage of $\Omega$ in $\H^3,$ is connected.  For some $(K, A)$, we have that $\tau_{\rho, x}:C_S(G) \rightarrow \tilde \Omega \subset \H^3$ is a $(K, A)$-quasi-isometry from $C_S(G)$ to $\tilde \Omega$ with the intrinsic metric.  In particular, any geodesic $l$ in $\mathcal{S}$ connecting $g_-$ and $g_+$, the fixed points of $g$, maps to a $(K, A)$-quasi-geodesic in $\tilde \Omega$, with the intrinsic metric.  Then, $\tau_{\rho,x}(l)$ lies in a $R=R_\Omega(K, A)$-neighborhood of $\textup{Ax}(g)$ in the intrinsic metric and also with the extrinsic metric, where, $\textup{Ax}(g)$ is the axis of $\rho(g)$ in $\H^3.$
If $x,y$ lie on $\tau_{\rho, x}(l)$ and if $\pi$ denotes the closest point projection onto $\textup{Ax}(g)$ in $\tilde \Omega$, then
$$
d_{\tilde \Omega}(x, y) \leq d_{\tilde \Omega}(\pi(x),\pi( y)) + 2R = d_{\H^3}(\pi(x), \pi(y)) + 2R \leq d_{\H^3} (x, y)+4R
$$
This implies that $\tau_{\rho, x}(l)$ is a $(K, A+4R)$-quasi-geodesic in $\tilde \Omega$ with the extrinsic metric.  Hence $\rho$ is $(K, A+4R)$-separable-stable.
\end{proof}

\begin{lemma} \label{notprimstab}
Let $\rho$ be a discrete faithful representation such that $\rho(g)$ is hyperbolic for all separable elements $g$.  If $\rho$ is not separable-stable, then there exists a sequence of separable elements $g_i$ such that the endpoints of $\textup{Ax}(\rho(g_i))$ converge to a single point $z$ in $\partial \H^3$ but the points $g_i^+$ and $g_i^-$ in $\partial C_S(G)$ converge to distinct points $z^+$ and $z^-$ in $\partial C_S(G)$.
\end{lemma}

\begin{proof}
By Lemma \ref{compact}, since $\rho$ is not separable-stable, the set of geodesics homotopic to separable curves is not contained in any compact set of $N_\rho$.  Let $\{\gamma_i\}$ be a sequence of separable geodesics such that $\{\gamma_i\}$ is not contained in any compact set.  Recall that the image of the Cayley graph under $\tau_{\rho, x}$ in $N_\rho$ has only one vertex, $v$.  Choose $D_i$ approaching infinity such that $\gamma_i$ does not lie in a ball of radius $D_i$ around $v$.

Fix a set of lifts $\widetilde \gamma_i$ of $\gamma_i$.  Then $\widetilde \gamma_i$ is an axis for $\rho(g_i)$ for some separable element $g_i$.  Let $l_i$ be a geodesic in the Cayley graph connecting $g_i^+$ and $g_i^-$.  There exists a vertex $v_i$ on $\tau_{\rho,x}(l_i)$ such that the distance to $\widetilde \gamma_i$ is at least $D_i$.  Shift $l_i$ to $l_i':=v_i^{-1} \cdot l_i$.  Then, $l_i'$ passes through $e$ the identity element and connects the fixed points of $v_i^{-1}g_iv_i$, which is still a separable element.  If $\widetilde \gamma_i'$ is $\rho(v_i^{-1}) \cdot \widetilde \gamma_i$, the distance from $x$ to $\widetilde \gamma_i'$ is $D_i$.  This implies that, up to subsequence, the endpoints of $\widetilde \gamma_i$ approach a single point $z$ on $\partial \H^3$.  There exists $z^+$ and $z^-$ on $\partial C_S(G)$ such that up to subsequence, $v_i^{-1}g_iv_i^+ \rightarrow z^+$ and $v_i^{-1}g_iv_i^- \rightarrow z^-$.  Since each $l_i'$ passes through the $e$, $z^+$ and $z^-$ are distinct.
\end{proof}

\begin{proposition} \label{pspoints}
Let $\rho$ be a discrete and faithful representation such that $N_\rho = \H^3/\rho(G)$ is homeomorphic to the interior of $M$ satisfying one of the following two conditions.
\begin{enumalph}
\item $\rho$ is a geometrically finite representation with one cusp associated to a Masur domain curve or
\item $\rho$ is a purely hyperbolic representation where the end corresponding to the exterior boundary is geometrically infinite and all other ends are convex cocompact.
\end{enumalph}
Then $\rho$ is separable-stable.
\end{proposition}

\begin{proof}
Let $\lambda \subset S$ be the cusp curve if $\rho$ is of type (a) or the ending lamination if $\rho$ is of type (b).  If $\rho$ is not separable-stable, by Lemma \ref{notprimstab} there exists a sequence of separable elements $g_i$ such that the endpoints of $\textup{Ax}(\rho(g_i))$ converge to a single point $z$ in $\partial \H^3$ but the points $g_i^+$ and $g_i^-$ in $\partial C_S(G)$ converge to distinct points $z^+$ and $z^-$ in $\partial C_S(G)$.  In particular, $z^+$ and $z^-$ are identified under the Cannon--Thurston map for $\tau_{\rho,x}$.

Consider our fixed convex cocompact representation $\sigma: G \rightarrow \pslc$ used to define the Whitehead graph.  If $\bar\tau_{\sigma, x}$ is the Cannon--Thurston map for $\tau_{\sigma,x}$, using the results of Floyd and Mj (see Section \ref{ctmaps}), $\bar\tau_{\sigma, x}(z^+)$ and $\bar\tau_{\sigma, x}(z^-)$ are either
\begin{enumalph}
\item endpoints of $\textup{Ax}(\sigma(g))$ where $\rho(g)$ is parabolic or
\item endpoints of a leaf of the ending lamination or ideal endpoints of a complementary polygon
\end{enumalph}
where the first case occurs if $\rho$ is of type (a) and the second case occurs if $\rho$ is of type (b) (as in the statement of the proposition).

Let $\mu_\infty \subset \Lambda(\sigma(G)) \times \Lambda(\sigma(G))$ be the set of limit points of $\{\mu_{g_i}\}.$  Then $\mu_\infty$ is $\sigma(G)$-invariant and invariant under switching the two factors.  Moreover, $(\bar\tau_{\sigma, x}(z^+), \bar\tau_{\sigma, x}(z^-))$ lies in $\mu_\infty$.

Since ending laminations lie in the Masur domain (see Section \ref{ctmaps}), we can choose $\alpha$ a system of meridians such that $\lambda$ is in tight position with respect to $\alpha$ (see Section \ref{topinterpret}).  We first claim that $\Gamma_\alpha(\mu_\lambda)$ is contained in $\Gamma_\alpha(\mu_\infty)$.  In case (a) this is obvious as $\mu_\lambda$ is exactly the $\sigma(G)$ translates of $(\bar\tau_{\sigma, x}(z^+), \bar\tau_{\sigma, x}(z^-))$.  In case (b), let $L$ be the geodesic connecting $\bar\tau_{\sigma, x}(z^+)$ and $\bar\tau_{\sigma, x}(z^-)$ and $l$ be the geodesic in $\Omega(\sigma(G))$ that is a leaf of the preimage of the ending lamination with one endpoint $\bar\tau_{\sigma, x}(z^+)$.  Let $w$ be the other endpoint of $l$.  Recall that to define the Whitehead graph we fixed a system of meridians $\alpha$ on $\partial_C N_\sigma$ that bound discs $D$.  Let $\widetilde \alpha_i$ be a lift of one of the meridians $\alpha_i$ with the following property.  $\partial \H^3 - \widetilde \alpha_i$ has two components $W_1$ and $W_2$ such that $\bar\tau_{\sigma, x}(z^+)$ lies in $W_1$ and $\bar\tau_{\sigma, x}(z^-)$ and $w$ lie in $W_2$.  Let $r$ be the ray of $l$ that starts at $\widetilde \alpha_i$ and ends at $\bar\tau_{\sigma, x}(z^+)$ (see Figure \ref{notprimstabpic}).

\begin{figure}[ht!]
\begin{center}
 \includegraphics[width=80mm]{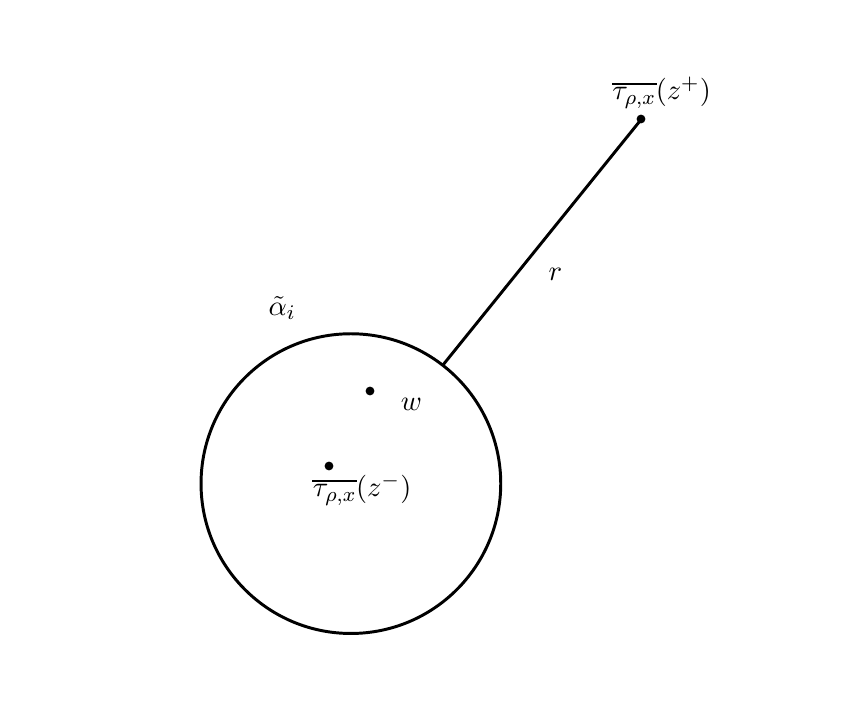}
\caption[A ray]{$r$ is a ray in $\partial \H^3$ that starts at $\widetilde \alpha_i$ and ends at $\overline{\tau_{\rho,x}}(z^+)$.}
\label{notprimstabpic}
\end{center}
\end{figure}

Edges in $\Gamma_\alpha(\mu_\lambda)$ correspond to homotopy classes of arcs of $\lambda$ connecting the components of $\alpha$.  Since $\lambda$ is minimal, $r'$ the image of $r$ in $\partial_{ext}M$ is dense in $\lambda$.  Then, for any edge in $\Gamma_\alpha(\mu_\lambda)$ there is an arc of $r'$ corresponding to that edge.
Let $(U, gV)$ be an edge of $\Gamma_\alpha(\mu_\lambda)$ and $r_0$ the arc of $r'$ corresponding to that edge, i.e., there is a lift $\widetilde r_0$ of $r_0$ with one endpoint on $\partial U$ and the other on $\partial gV$.  This means that there is a translate $h \cdot r$ of $r$ such that $h \cdot r$ intersects $\partial U$ and $\partial gV$.  This implies that $h \cdot \bar\tau_{\sigma, x}(z^+)$ lies in $U$ or $gV$.  Without loss of generality assume that $h \cdot \bar\tau_{\sigma, x}(z^+)$ lies in $gV$.  Then it suffices to show that $h \cdot \bar\tau_{\sigma, x}(z^-)$ lies in $U$.  Since $r$ intersects $U$ and is in tight position with respect to $\alpha$, the translate $h \cdot \widetilde \alpha_i$ must lie inside $U$. This implies that $h \cdot \bar\tau_{\sigma, x}(z^-)$ lies in $U$ (see Figure \ref{notprimstabpic2}).  This completes the proof of the claim.

\begin{figure}[ht!]
\begin{center}
 \includegraphics[width=60mm]{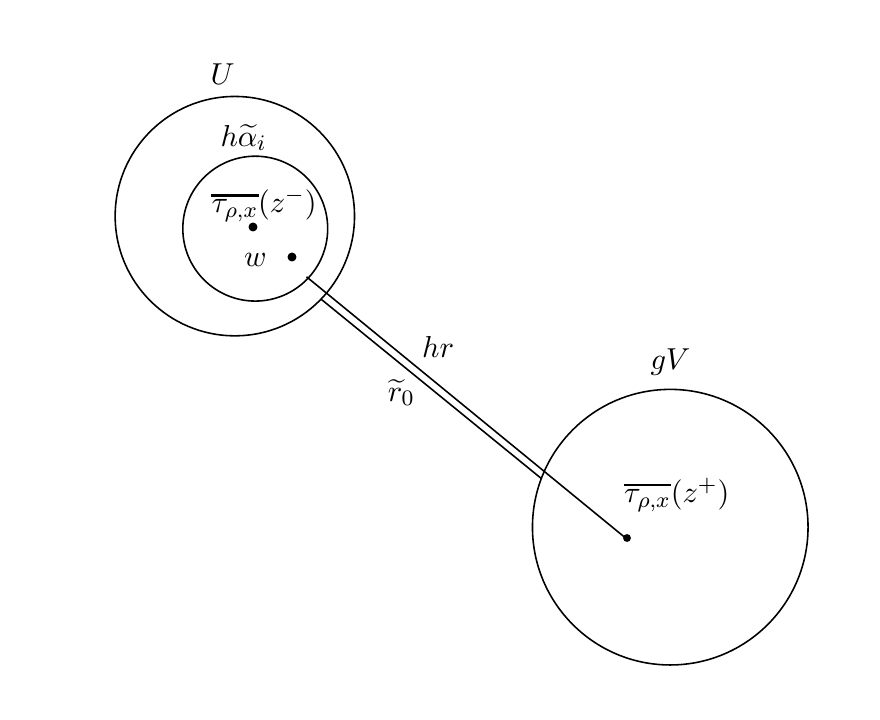}
\caption[A leaf]{Since $\lambda$ is in tight position with respect to $\alpha$ there is a leaf with one endpoint in $U$ and the other in $gV$.}
\label{notprimstabpic2}
\end{center}
\end{figure}

Secondly, we observe that $\Gamma_\alpha(\mu_\lambda)$ is a finite graph.  In case (a) this is obvious as $\lambda$ is a closed curve.  For case (b) recall the topological interpretation of the Whitehead graph (see Section \ref{topinterpret}), where edges in the Whitehead graph correspond to homotopy classes of arcs of $\lambda$ with endpoints on $\partial D$, relative to those endpoints. Since there can only be finitely many homotopy classes of arcs with endpoints on $\partial D$ that can be realized disjointly on $\partial_{ext}M$, there can only be finitely many edges in $\Gamma_\alpha(\mu_\lambda)$.

Since $\Gamma_\alpha(\mu_\lambda)$ is a finite graph contained in $\Gamma_\alpha(\mu_\infty)$ for $i$ large enough, $\Gamma_\alpha(\mu_{g_i})$ contains $\Gamma_\alpha(\mu_\lambda)$ as a subgraph.  Notice that any vertex in $\Gamma_\alpha(\mu_{g_i})$ is also a vertex in $\Gamma_\alpha(\mu_\lambda)$.  We claim that this implies that $\Gamma_\alpha(\mu_{g_i})$ must be strongly connected and without a strong cutpoint by Proposition \ref{masurdomain}.  Indeed, since any vertex in $\Gamma_\alpha(\mu_{g_i})$ is part of a nontrivial cycle in $\Gamma_\alpha(\mu_\lambda)$, it is part of the same nontrivial cycle in $\Gamma_\alpha(\mu_{g_i})$.   If $\Gamma_\alpha(\mu_{g_i})$ had a strong cutpoint $v$, then the component of $\Gamma_\alpha(\mu_\lambda)$ containing $v$ would either be not strongly connected or also have a strong cutpoint.  When $M$ is not the connect sum of two trivial $I$-bundles over closed surfaces, this contradicts Proposition \ref{primnsc}.

For the case when $M$ is the connect sum of two trivial $I$-bundles over closed surfaces, we first claim that each edge of $\Gamma_\alpha(\mu_\lambda)$ ``intersects'' every essential annulus, $A$, contained in each trivial $I$-bundle in the following sense.  Let $M= (S \times I) \# (T \times I)$.  Suppose that $A$ is an annulus in $S \times I$ or $T \times I$.  Let $\partial A = c_1 \sqcup c_2$.  In defining the Whitehead graph, we fixed lifts $\widetilde S$ and $\widetilde T$ of $S$ and $T$.  If we take a lift $\widetilde c_1$ of $c_1$ and the lift $\widetilde c_2$ of $c_2$ with the same endpoints as $\widetilde c_1$, then $\widetilde c_1 \cup \widetilde c_2$ forms a loop in $\partial \H^3$.
 We will say that an edge $e = (U_i, gU_j)$ intersects $A$ if there exists lifts $\widetilde c_1$ and $\widetilde c_2$ in either $\widetilde S$ or $\widetilde T$ as above such that $U_i$ and $gU_j$ lie in different components of $\partial \H^3 - (\widetilde c_1 \cup \widetilde c_2)$.

If $\lambda$ is a Masur domain lamination, then it intersects every essential annulus.  Using the topological interpretation of the Whitehead graph (see Section \ref{topinterpret}), $\Gamma_\alpha(\mu_\lambda)$ ``intersects'' every essential annulus in the sense above.

Since $\Gamma_\alpha(\mu_\lambda)$ is a finite graph contained in $\Gamma_\alpha(\mu_\infty)$ for $i$ large enough, $\Gamma_\alpha(\mu_{g_i})$ ``intersects'' any essential annulus contained in one of the two trivial $I$-bundles.  This implies that the geodesic representative $\gamma_i$ of $g_i$ intersects any essential annulus contained in one of the two factors, a contradiction.
\end{proof}

The assumption in Proposition \ref{pspoints} that each end corresponding to a component of the interior boundary is incompressible is necessary as the following proposition shows.  Recall that $\pi_1(M) = \pi_1(S_1) * \ldots * \pi_1(S_k) * F_j$, where $S_i$ is a closed surface and $F_j$ is the free group on $j$ elements.
\begin{proposition}
If $[\rho]$ lies in $\partial AH(M)$ such that $\rho|_{\pi_1(S_i)}$ is not convex cocompact, then $[\rho]$ does not lie in $\mathcal{SS}(M)$.  Moreover, if $\rho|_{\pi_1(S_i)}$ is not convex cocompact, then $[\rho]$ cannot lie in a domain of discontinuity
\end{proposition}
\begin{proof}
If $\rho|_{\pi_1(S_i)}$ has a cusp, then there is a separable element that maps to a parabolic element.  So $\rho$ cannot be separable-stable.  If $\rho|_{\pi_1(S_i)}$ is geometrically infinite, then in $N_\rho$ there is a sequence of separable geodesics exiting every compact set, so $\rho$ cannot be separable-stable by Lemma \ref{compact}.

To see the second statement, suppose that $\rho(g)$ is parabolic for a separable element $g$.  Then, there exists a sequence of representations $\rho_i$ in $\mathcal{X}(M)$ such that $\rho_i(c)$ is elliptic of finite order, $n_i$ (see \cite{lee1} Lemma 15).  Since $g$ is separable, there exists a nontrivial splitting, $G=G_1 * G_2$ such that $g$ lies in $G_1$.  The automorphism $f_{n_i}$ that restricts to conjugation by $g^{n_i}$ on $G_1$ and restricts to the identity on $G_2$ fixes $\rho_{i}$.  In particular, each $\rho_i$ has an infinite stabilizer, so, the limit $[\rho]$ cannot lie in any domain of discontinuity.

Now suppose that $\rho(g)$ is not parabolic for any separable element $g$ but $\rho|_{\pi_1(S_i)}$ is geometrically infinite for some $i$.  Then, we will describe a sequence of representations $\rho_k$ approaching $\rho$ such that there exists a separable curve $g_k$ with $\rho_k(g_k)$ parabolic.  Since each $\rho_k$ cannot lie in a domain of discontinuity, neither can $\rho$.  To find such a sequence first observe that $\rho|_{\pi_1(S_i)}$ is purely hyperbolic.  By the covering theorem (\cite{can1}), it can have only one geometrically infinite end and so it lies in the closure of a Bers slice, $\overline{B}.$  Let $\lambda$ be its ending lamination, and let $\gamma_j$ be a sequence of simple closed curves on $S_i$ approaching $\lambda$.  Define a sequence of representations $\rho_k$ satisfying the following:
 \begin{itemize}
 \item $\rho_k|_{\pi_1(S_i)}$ lies in $\overline{B}$,
 \item $\rho_k(\gamma_j)$ is parabolic and
 \item $\rho_k|_{\pi_1(S_1) * \ldots *\widehat{\pi_1(S_i)} * \ldots* \pi_1(S_k) * F_j}=\rho|_{\pi_1(S_1) * \ldots *\widehat{\pi_1(S_i)} * \ldots* \pi_1(S_k) * F_j}.$
 \end{itemize}
  Then, as $\overline{B}$ is compact, up to subsequence $\rho_k|_{\pi_1(S_i)}$ converges to some $\rho'$ in $\overline{B}$.  As the length function is continuous on $AH(S_i \times I)$ (\cite{bro}), the length of $\lambda$ in $\rho'$ must be zero. In particular, $\lambda$ must be an ending lamination for $\rho'$.  By the ending lamination theorem (\cite{bro-can-min1}), possibly after conjugating, $\rho_k|_{\pi_1(S_i)}$ must converge to $\rho|_{\pi_1(S_i)}$.  On the other factors of $\pi_1(M),$ by construction $\rho_k$ converges to $\rho$.  Hence, $\rho_k$ converges to $\rho$ where $\rho_k$ has a separable curve pinched.  Since $[\rho_k]$ cannot lie in a domain of discontinuity, neither can $[\rho].$
\end{proof}

\section{Other homeomorphism types in $AH(M)$}\label{nonhomeo}

So far we have found separable-stable points on $\partial AH(M)$ that have the same homeomorphism type as $M$. In this section, we show that if $M$ is a large compression body, then there exists $M'$ homotopy equivalent but not homeomorphic to $M$ such that the for each component $C$ of the interior of $AH(M)$ corresponding to $M'$,  no point on $\partial C$ is separable-stable, even though every point in $C$ is separable-stable.

\begin{proposition} \label{otherhomeo}
Suppose that $M'$ is homotopy equivalent to $M$ such that
\begin{enumalph}
\item $M'$ is not homeomorphic to $M$
\item for each compressible component $B$ of $\partial M',$ the subgroup $i_*(\pi_1(B))$ is a free factor of $\pi_1(M').$
\end{enumalph}
If $C$ is a component of the interior of $AH(M)$ corresponding to $M'$, then $\overline{C}- C$ has no separable stable points.
\end{proposition}

\begin{proof}
Let $[\rho]$ be a purely hyperbolic point in $\overline{C}- C$.  Then $\rho$ is the algebraic limit of $\rho_i$ in $C$ such that $N_{\rho_i}$ is homeomorphic to the interior of $M'$.  Since $\rho$ has no parabolics, $\rho_i$ converges to $\rho$ geometrically (see \cite{bro-bro-can-min}).  Then, $N_\rho$ is homeomorphic to $N_{\rho_i}$ (\cite{can-min}).  Since each boundary component of $M'$ maps to a proper factor of a free decomposition of $\pi_1(M)$, there is a boundary component $B$ of $M'$ such that the end corresponding to $B$ is geometrically infinite.  Then, there exists a sequence of simple closed curves on $B$ whose geodesic representatives leave every compact set of $N_\rho$.  By Lemma \ref{compact}, $\rho$ cannot be separable-stable, as any simple closed curve on $B$ is separable.  Since purely hyperbolic points are dense in $\overline{C}-C$ (\cite{can-her}, Lemma 4.2, \cite{nam-sou}, \cite{ohs}), this completes the proof.
\end{proof}

If $M$ is a large compression body, then there always exists such an $M'$.  Any $M'$ homotopy equivalent to $M$ with more than one compressible boundary component will suffice.  Suppose that $B$ and $B'$ are two compressible boundary components of $M'.$   Let $m$ be a meridian in $B'$ bounding a disc $D$.  If $D$ separates $M'$ into $M'_1$ and $M'_2$, then $\pi_1(M) \cong \pi_1(M'_1) * \pi_1(M'_2)$ and $\pi_1(B)$ lies in one of the two factors.  If $D$ is non-separating and $M''=M'-D$, then $\pi_1(M) \cong \pi_1(M'')*_{\{1\}} \cong \pi_1(M'')*\mathbb{Z}$ and $\pi_1(B)$ lies in the first factor.


\begin{thebibliography}
{100}
{\footnotesize
\bibitem{ago} I. Agol, Tameness of hyperbolic 3-manifolds, preprint available at arxiv:math\textbackslash 0405568v1
\bibitem{ben-pet} R. Benedetti\ and\ C. Petronio, {\it Lectures on hyperbolic geometry}, Universitext, Springer, Berlin, 1992.
\bibitem{bon1} F. Bonahon, Bouts des vari\'et\'es hyperboliques de dimension $3$, Ann. of Math. (2) {\bf 124} (1986), no.~1, 71--158.
\bibitem{bro} J. F. Brock, Continuity of Thurston's length function, Geom. Funct. Anal. {\bf 10} (2000), no.~4, 741--797. MR1791139 (2001g:57028)
\bibitem{bro-bro-can-min} J. F. Brock, K. Bromberg, R.D. Canary and Y. Minsky, Local topology in deformation spaces of hyperbolic 3-manifolds, Geom. Topol. {\bf 15} (2011), no.~2, 1169--1224. MR2831259
\bibitem{bro-can-min1} J. F. Brock, R.D. Canary and Y. Minsky, The Classification of Kleinian Surface Groups II: The Ending Lamination Conjecture, Ann. of Math. {\bf 176} (2012), 1--149
\bibitem{cal-gab} D. Calegari\ and\ D. Gabai, Shrinkwrapping and the taming of hyperbolic 3-manifolds, J. Amer. Math. Soc. {\bf 19} (2006), no.~2, 385--446. MR2188131 (2006g:57030)
\bibitem{can2} R. D. Canary, Ends of hyperbolic $3$-manifolds, J. Amer. Math. Soc. {\bf 6} (1993), no.~1, 1--35. MR1166330 (93e:57019)
\bibitem{can1} R. D. Canary, A covering theorem for hyperbolic $3$-manifolds and its applications, Topology {\bf 35} (1996), no.~3, 751--778. MR1396777 (97e:57012)
\bibitem{can-survey}R.D. Canary, Dynamics on Character Varieties: A survey, in preparation
\bibitem{ceg} R.D. Canary,\ D.B.A. Epstein,\ and\ P. Green, Notes on notes of Thurston in {\it Analytical and geometric aspects of hyperbolic space (Coventry/Durham, 1984),} volume 111 of {\it London Math. Soc. Lecture Note Ser.}, 3-92. Cambridge Univ Press, Cambridge, 1987
\bibitem{can-her} R. D. Canary\ and\ S. Hersonsky, Ubiquity of geometric finiteness in boundaries of deformation spaces of hyperbolic 3-manifolds, Amer. J. Math. {\bf 126} (2004), no.~6, 1193--1220. MR2102392 (2005k:57032)
\bibitem{can-mag} R.D. Canary\ and\ A. Magid, Dynamics on $\pslc$-character varieties: $3$-manifolds with toroidal boundary components, preprint available at arXiv:1110.6567
\bibitem{can-mcc}R. D. Canary\ and\ D. McCullough, Homotopy equivalences of 3-manifolds and deformation theory of Kleinian groups, Mem. Amer. Math. Soc. {\bf 172} (2004), no.~812.
\bibitem{can-min} R. D. Canary\ and\ Y. N. Minsky, On limits of tame hyperbolic $3$-manifolds, J. Differential Geom. {\bf 43} (1996), no.~1, 1--41. MR1424418 (98f:57021)
\bibitem{can-sto} R.D. Canary\ and P. Storm, Moduli spaces of hyperbolic 3-manifolds and dynamics on character varieties,  Commentarii Mathematici Helvetici, to appear
\bibitem{can-thu} J. W. Cannon\ and\ W. P. Thurston, Group invariant Peano curves, Geom. Topol. {\bf 11} (2007), 1315--1355. MR2326947 (2008i:57016)
\bibitem{flo} W. J. Floyd, Group completions and limit sets of Kleinian groups, Invent. Math. {\bf 57} (1980), no.~3, 205--218. MR0568933 (81e:57002)
\bibitem{gru} I. A. Grushko, On the bases of a free product of groups. Matematicheskii Sbornik, 8:169--182, 1940.
\bibitem{har-str} T. Hartnick\ and\ T. Strubel, Cross ratios, translation lengths and maximal representations, Geom. Dedicata {\bf 161} (2012), 285--322. MR2994044
\bibitem{jeo-kim} W. Jeon\ and\ I. Kim, Primitive stable representations of geometrically infinite handlebody hyperbolic 3-manifolds, C. R. Math. Acad. Sci. Paris {\bf 348} (2010), no.~15-16, 907--910. MR2677989 (2011i:57022)
\bibitem{kap} M. Kapovich, {\it Hyperbolic manifolds and discrete groups}, Progress in Mathematics, 183, Birkh\"auser Boston, Boston, MA, 2001. MR1792613 (2002m:57018)
\bibitem{kle-sou} G. Kleineidam\ and\ J. Souto, Algebraic convergence of function groups, Comment. Math. Helv. {\bf 77} (2002), no.~2, 244--269. MR1915041 (2003d:57037)
\bibitem{kur} A. Kurosch, Die Untergruppen der freien Produkte von beliebigen Gruppen, Math. Ann. {\bf 109} (1934), no.~1, 647--660. MR1512914
\bibitem{lab2} F. Labourie, Cross ratios, Anosov representations and the energy functional on Teichm\"uller space, Ann. Sci. \'Ec. Norm. Sup\'er. (4) {\bf 41} (2008), no.~3, 437--469. MR2482204 (2010k:53145)
\bibitem{lec} C. Lecuire, An extension of the Masur domain, in {\it Spaces of Kleinian groups}, 49--73, London Math. Soc. Lecture Note Ser., 329 Cambridge Univ. Press, Cambridge. MR2258744 (2007h:57026)
\bibitem{lee1} M. Lee, Dynamics on the $\pslc$-character variety of a twisted $I$-bundle, preprint.  	arXiv:1103.3479, 2011
\bibitem{thesis} M. Lee, Dynamics on $\pslc$-character varieties of certain hyperbolic $3$-manifolds, Ph.D. Thesis, University of Michigan, 2012, available at http://hdl.handle.net/2027.42/94050
\bibitem{mas} H. Masur, Measured foliations and handlebodies, Ergodic Theory Dynam. Systems {\bf 6} (1986), no.~1, 99--116. MR0837978 (87i:57011)
\bibitem{mcc} D. McCullough, Compact submanifolds of $3$-manifolds with boundary, Quart. J. Math. Oxford Ser. (2) {\bf 37} (1986), no.~147, 299--307.
\bibitem{mcc-mil} D. McCullough\ and\ A. Miller, Homeomorphisms of $3$-manifolds with compressible boundary, Mem. Amer. Math. Soc. {\bf 61} (1986), no.~344, {\rm xii}+100 pp. MR0840832 (87i:57013)
\bibitem{min} Y. Minsky, On dynamics of $\out(F_n)$ on ${\rm \textup{PSL}}\sb 2\mathbb{C}$ characters, preprint, arXiv:0906.3491v1
\bibitem{mj} M. Mj, Cannon-Thurston maps for Kleinian groups, preprint available at arXiv:1002.0996
\bibitem{nam-sou} H. Namazi\ and\ J. Souto, Non-realizability and ending laminations: proof of the density conjecture, Acta Math. {\bf 209} (2012), no.~2, 323--395. MR3001608
\bibitem{ohs} K. Ohshika, Realising end invariants by limits of minimally parabolic, geometrically finite groups, Geom. Topol. {\bf 15} (2011), no.~2, 827--890. MR2821565 (2012g:57035)
\bibitem{ota} J.-P. Otal,Courants g\'eod\'esiques et produits libres, Th\'ese d’Etat, Universit\'e de Paris-Sud, Orsay, 1988.
\bibitem{sco} G. P. Scott, Compact submanifolds of $3$-manifolds, J. London Math. Soc. (2) {\bf 7} (1973), 246--250.
\bibitem{sta} J. R. Stallings, Whitehead graphs on handlebodies, in {\it Geometric group theory down under (Canberra, 1996)}, 317--330, de Gruyter, Berlin. MR1714852 (2001i:57028)
\bibitem{sul} D. Sullivan, Quasiconformal homeomorphisms and dynamics. II. Structural stability implies hyperbolicity for Kleinian groups, Acta Math. {\bf 155} (1985), no.~3-4, 243--260.
\bibitem{thu} W. P. Thurston, {\it Three--dimensional geometry and topology. Vol. 1}, Princeton Mathematical Series, 35, Princeton Univ. Press, Princeton, NJ, 1997. MR1435975 (97m:57016)
\bibitem{whi1} J. H. C. Whitehead, On Certain Sets of Elements in a Free Group, Proc. London Math. Soc. {\bf S2-41} no.~1, 48. MR1575455
\bibitem{whi2} J. H. C. Whitehead, On equivalent sets of elements in a free group, Ann. of Math. (2) {\bf 37} (1936), no.~4, 782--800. MR1503309
\bibitem{wie}A. Wienhard, The action of the mapping class group on maximal representations, Geom. Dedicata {\bf 120} (2006), 179--191. MR2252900 (2008g:20112)
}
\end{thebibliography}
\end{document}